\documentclass[11pt,a4paper]{amsart}

\usepackage[T1]{fontenc}
\usepackage[utf8]{inputenc}
\usepackage{lmodern}
\usepackage{microtype}
\usepackage[a4paper,margin=2.65cm]{geometry}
\usepackage{amsmath,amssymb,amsthm,mathtools}
\usepackage{enumitem}
\usepackage{needspace}
\usepackage{xcolor}
\usepackage{cite}
\usepackage{tikz}
\usetikzlibrary{arrows.meta}
\usepackage{hyperref}

\hypersetup{
  colorlinks=true,
  linkcolor=blue!55!black,
  citecolor=blue!55!black,
  urlcolor=blue!55!black,
  pdftitle={Surface-to-volume homogenization of nonlinear Steklov eigencouples with nonhomogeneous Orlicz growth in perforated domains},
  pdfauthor={Hermann Douanla and Joel Fotso Tachago and Hubert Nnang and Jean Louis Woukeng}
}

\allowdisplaybreaks
\setlist[enumerate]{leftmargin=2.2em,itemsep=0.25em,topsep=0.35em}
\setlist[itemize]{leftmargin=2.0em,itemsep=0.25em,topsep=0.35em}
\numberwithin{equation}{section}

\newtheorem{maintheorem}{Theorem}

\newtheorem{theorem}{Theorem}[section]
\newtheorem{proposition}[theorem]{Proposition}
\newtheorem{lemma}[theorem]{Lemma}

\theoremstyle{definition}
\newtheorem{definition}[theorem]{Definition}
\newtheorem{remark}[theorem]{Remark}

\title[Surface-to-volume Steklov homogenization]
{Surface-to-volume homogenization of nonlinear Steklov eigencouples with nonhomogeneous Orlicz growth in perforated domains}

\author[H. Douanla]{Hermann Douanla}
\address{Department of Mathematics, Faculty of Science, University of Yaound\'e I, P.O. Box 812, Yaound\'e, Cameroon}
\email{hermann.douanla@facsciences-uy1.cm}

\author[J. Fotso Tachago]{Joel Fotso Tachago}
\address{Department of Mathematics, Higher Teachers Training College, University of Bamenda, P.O. Box 39, Bambili, Cameroon}
\email{fotsotachago@yahoo.fr}

\author[H. Nnang]{Hubert Nnang}
\address{Department of Mathematics, Higher Teachers Training College, University of Yaound\'e I, P.O. Box 47, Yaound\'e, Cameroon}
\email{hnnang@yahoo.fr}

\author[J. L. Woukeng]{Jean Louis Woukeng}
\address{Department of Mathematics and Computer Science, University of Dschang, P.O. Box 67, Dschang, Cameroon}
\email{jeanlouis.woukeng@univ-dschang.org}

\begin{document}

\begin{abstract}
We study normalized principal Steklov eigencouples on periodically perforated domains when both the bulk energy and the boundary normalization have general, possibly nonhomogeneous, Orlicz growth. The surface term on the hole boundaries carries the critical surface-to-volume scaling. For each fixed mass, the constrained minima converge along the whole family, while selected minimizers, fluxes and Lagrange multipliers converge along subsequences to the corresponding objects of a homogenized problem. In the limit, the bulk operator is given by a convex cell problem, and the boundary term becomes a volume term weighted by the total surface density of the reference hole. The key tool is a strong approximation of the microscopic traces by cell averages, measured in the scaled surface modular.

\vspace{9pt}

\end{abstract}

\maketitle

\section{Introduction}
\label{sec:introduction}

\subsection{Model and motivation}
\label{subsec:motivation}

Let \(N\geq2\) and \(Y=(0,1)^{N}\). The geometry satisfies the following assumption.

\medskip
\noindent
\((H_0)\)\quad
The set \(\Omega\subset\mathbb{R}^{N}\) is a bounded Lipschitz domain, and
\(T\Subset Y\) is a nonempty open reference hole with boundary of
class \(C^{1,1}\), such that \(Y^{*}:=Y\setminus\overline{T}\) is connected.

\medskip
\noindent
We denote by \(\nu\) the outward unit normal to \(Y^{*}\) on \(\partial T\),
which points into the hole. For \(0<\varepsilon<1\), define the set of complete cells \(K_{\varepsilon}\), the union of holes \(T_{\varepsilon}\), the perforated domain \(\Omega_{\varepsilon}\) and the union of internal boundaries \(\Gamma_{\varepsilon}\) by
\[
\begin{aligned}
K_{\varepsilon}&=\bigl\{k\in\mathbb{Z}^{N}:\varepsilon(k+\overline{Y})\subset\Omega\bigr\},
&\qquad
T_{\varepsilon}&=\textstyle\bigcup_{k\in K_{\varepsilon}}\varepsilon(k+T),
\\
\Omega_{\varepsilon}&=\Omega\setminus\overline{T_{\varepsilon}},
&\qquad
\Gamma_{\varepsilon}&=\textstyle\bigcup_{k\in K_{\varepsilon}}\varepsilon(k+\partial T).
\end{aligned}
\]

By construction, only complete cells are perforated; no hole intersects
\(\partial\Omega\), and the remaining boundary layer is unperforated
(Figure~\ref{fig:geometry}).
\begin{figure}[t]
\centering
\begin{tikzpicture}[scale=1]
\begin{scope}[shift={(-6.3,-1.25)},scale=2.5]
\fill[gray!8] (0,0) rectangle (1,1);
\draw[line width=0.6pt] (0,0) rectangle (1,1);
\fill[gray!35] (0.5,0.5) ellipse (0.27 and 0.2);
\draw[line width=0.6pt] (0.5,0.5) ellipse (0.27 and 0.2);
\node at (0.56,0.5) {\(T\)};
\node at (0.15,0.86) {\(Y^{*}\)};
\node at (0.84,0.31) {\(\partial T\)};
\draw[-{Stealth[length=2.2mm]},line width=0.6pt] (0.23,0.5) -- (0.42,0.5);
\node[left] at (0.23,0.5) {\(\nu\)};
\node[below] at (0.5,0) {\(Y=(0,1)^{N}\)};
\end{scope}
\begin{scope}
\draw[line width=0.7pt,fill=gray!8] (0,0) ellipse (2.7 and 1.75);
\draw[gray!45,line width=0.25pt] (-2.380,-0.680) rectangle (-2.040,-0.340);
\draw[gray!45,line width=0.25pt] (-2.380,-0.340) rectangle (-2.040,0.000);
\draw[gray!45,line width=0.25pt] (-2.380,0.000) rectangle (-2.040,0.340);
\draw[gray!45,line width=0.25pt] (-2.380,0.340) rectangle (-2.040,0.680);
\draw[gray!45,line width=0.25pt] (-2.040,-1.020) rectangle (-1.700,-0.680);
\draw[gray!45,line width=0.25pt] (-2.040,-0.680) rectangle (-1.700,-0.340);
\draw[gray!45,line width=0.25pt] (-2.040,-0.340) rectangle (-1.700,0.000);
\draw[gray!45,line width=0.25pt] (-2.040,0.000) rectangle (-1.700,0.340);
\draw[gray!45,line width=0.25pt] (-2.040,0.340) rectangle (-1.700,0.680);
\draw[gray!45,line width=0.25pt] (-2.040,0.680) rectangle (-1.700,1.020);
\draw[gray!45,line width=0.25pt] (-1.700,-1.020) rectangle (-1.360,-0.680);
\draw[gray!45,line width=0.25pt] (-1.700,-0.680) rectangle (-1.360,-0.340);
\draw[gray!45,line width=0.25pt] (-1.700,-0.340) rectangle (-1.360,0.000);
\draw[gray!45,line width=0.25pt] (-1.700,0.000) rectangle (-1.360,0.340);
\draw[gray!45,line width=0.25pt] (-1.700,0.340) rectangle (-1.360,0.680);
\draw[gray!45,line width=0.25pt] (-1.700,0.680) rectangle (-1.360,1.020);
\draw[gray!45,line width=0.25pt] (-1.360,-1.360) rectangle (-1.020,-1.020);
\draw[gray!45,line width=0.25pt] (-1.360,-1.020) rectangle (-1.020,-0.680);
\draw[gray!45,line width=0.25pt] (-1.360,-0.680) rectangle (-1.020,-0.340);
\draw[gray!45,line width=0.25pt] (-1.360,-0.340) rectangle (-1.020,0.000);
\draw[gray!45,line width=0.25pt] (-1.360,0.000) rectangle (-1.020,0.340);
\draw[gray!45,line width=0.25pt] (-1.360,0.340) rectangle (-1.020,0.680);
\draw[gray!45,line width=0.25pt] (-1.360,0.680) rectangle (-1.020,1.020);
\draw[gray!45,line width=0.25pt] (-1.360,1.020) rectangle (-1.020,1.360);
\draw[gray!45,line width=0.25pt] (-1.020,-1.360) rectangle (-0.680,-1.020);
\draw[gray!45,line width=0.25pt] (-1.020,-1.020) rectangle (-0.680,-0.680);
\draw[gray!45,line width=0.25pt] (-1.020,-0.680) rectangle (-0.680,-0.340);
\draw[gray!45,line width=0.25pt] (-1.020,-0.340) rectangle (-0.680,0.000);
\draw[gray!45,line width=0.25pt] (-1.020,0.000) rectangle (-0.680,0.340);
\draw[gray!45,line width=0.25pt] (-1.020,0.340) rectangle (-0.680,0.680);
\draw[gray!45,line width=0.25pt] (-1.020,0.680) rectangle (-0.680,1.020);
\draw[gray!45,line width=0.25pt] (-1.020,1.020) rectangle (-0.680,1.360);
\draw[gray!45,line width=0.25pt] (-0.680,-1.360) rectangle (-0.340,-1.020);
\draw[gray!45,line width=0.25pt] (-0.680,-1.020) rectangle (-0.340,-0.680);
\draw[gray!45,line width=0.25pt] (-0.680,-0.680) rectangle (-0.340,-0.340);
\draw[gray!45,line width=0.25pt] (-0.680,-0.340) rectangle (-0.340,0.000);
\draw[gray!45,line width=0.25pt] (-0.680,0.000) rectangle (-0.340,0.340);
\draw[gray!45,line width=0.25pt] (-0.680,0.340) rectangle (-0.340,0.680);
\draw[gray!45,line width=0.25pt] (-0.680,0.680) rectangle (-0.340,1.020);
\draw[gray!45,line width=0.25pt] (-0.680,1.020) rectangle (-0.340,1.360);
\draw[gray!45,line width=0.25pt] (-0.340,-1.700) rectangle (0.000,-1.360);
\draw[gray!45,line width=0.25pt] (-0.340,-1.360) rectangle (0.000,-1.020);
\draw[gray!45,line width=0.25pt] (-0.340,-1.020) rectangle (0.000,-0.680);
\draw[gray!45,line width=0.25pt] (-0.340,-0.680) rectangle (0.000,-0.340);
\draw[gray!45,line width=0.25pt] (-0.340,-0.340) rectangle (0.000,0.000);
\draw[gray!45,line width=0.25pt] (-0.340,0.000) rectangle (0.000,0.340);
\draw[gray!45,line width=0.25pt] (-0.340,0.340) rectangle (0.000,0.680);
\draw[gray!45,line width=0.25pt] (-0.340,0.680) rectangle (0.000,1.020);
\draw[gray!45,line width=0.25pt] (-0.340,1.020) rectangle (0.000,1.360);
\draw[gray!45,line width=0.25pt] (-0.340,1.360) rectangle (0.000,1.700);
\draw[gray!45,line width=0.25pt] (0.000,-1.700) rectangle (0.340,-1.360);
\draw[gray!45,line width=0.25pt] (0.000,-1.360) rectangle (0.340,-1.020);
\draw[gray!45,line width=0.25pt] (0.000,-1.020) rectangle (0.340,-0.680);
\draw[gray!45,line width=0.25pt] (0.000,-0.680) rectangle (0.340,-0.340);
\draw[gray!45,line width=0.25pt] (0.000,-0.340) rectangle (0.340,0.000);
\draw[gray!45,line width=0.25pt] (0.000,0.000) rectangle (0.340,0.340);
\draw[gray!45,line width=0.25pt] (0.000,0.340) rectangle (0.340,0.680);
\draw[gray!45,line width=0.25pt] (0.000,0.680) rectangle (0.340,1.020);
\draw[gray!45,line width=0.25pt] (0.000,1.020) rectangle (0.340,1.360);
\draw[gray!45,line width=0.25pt] (0.000,1.360) rectangle (0.340,1.700);
\draw[gray!45,line width=0.25pt] (0.340,-1.360) rectangle (0.680,-1.020);
\draw[gray!45,line width=0.25pt] (0.340,-1.020) rectangle (0.680,-0.680);
\draw[gray!45,line width=0.25pt] (0.340,-0.680) rectangle (0.680,-0.340);
\draw[gray!45,line width=0.25pt] (0.340,-0.340) rectangle (0.680,0.000);
\draw[gray!45,line width=0.25pt] (0.340,0.000) rectangle (0.680,0.340);
\draw[gray!45,line width=0.25pt] (0.340,0.340) rectangle (0.680,0.680);
\draw[gray!45,line width=0.25pt] (0.340,0.680) rectangle (0.680,1.020);
\draw[gray!45,line width=0.25pt] (0.340,1.020) rectangle (0.680,1.360);
\draw[gray!45,line width=0.25pt] (0.680,-1.360) rectangle (1.020,-1.020);
\draw[gray!45,line width=0.25pt] (0.680,-1.020) rectangle (1.020,-0.680);
\draw[gray!45,line width=0.25pt] (0.680,-0.680) rectangle (1.020,-0.340);
\draw[gray!45,line width=0.25pt] (0.680,-0.340) rectangle (1.020,0.000);
\draw[gray!45,line width=0.25pt] (0.680,0.000) rectangle (1.020,0.340);
\draw[gray!45,line width=0.25pt] (0.680,0.340) rectangle (1.020,0.680);
\draw[gray!45,line width=0.25pt] (0.680,0.680) rectangle (1.020,1.020);
\draw[gray!45,line width=0.25pt] (0.680,1.020) rectangle (1.020,1.360);
\draw[gray!45,line width=0.25pt] (1.020,-1.360) rectangle (1.360,-1.020);
\draw[gray!45,line width=0.25pt] (1.020,-1.020) rectangle (1.360,-0.680);
\draw[gray!45,line width=0.25pt] (1.020,-0.680) rectangle (1.360,-0.340);
\draw[gray!45,line width=0.25pt] (1.020,-0.340) rectangle (1.360,0.000);
\draw[gray!45,line width=0.25pt] (1.020,0.000) rectangle (1.360,0.340);
\draw[gray!45,line width=0.25pt] (1.020,0.340) rectangle (1.360,0.680);
\draw[gray!45,line width=0.25pt] (1.020,0.680) rectangle (1.360,1.020);
\draw[gray!45,line width=0.25pt] (1.020,1.020) rectangle (1.360,1.360);
\draw[gray!45,line width=0.25pt] (1.360,-1.020) rectangle (1.700,-0.680);
\draw[gray!45,line width=0.25pt] (1.360,-0.680) rectangle (1.700,-0.340);
\draw[gray!45,line width=0.25pt] (1.360,-0.340) rectangle (1.700,0.000);
\draw[gray!45,line width=0.25pt] (1.360,0.000) rectangle (1.700,0.340);
\draw[gray!45,line width=0.25pt] (1.360,0.340) rectangle (1.700,0.680);
\draw[gray!45,line width=0.25pt] (1.360,0.680) rectangle (1.700,1.020);
\draw[gray!45,line width=0.25pt] (1.700,-1.020) rectangle (2.040,-0.680);
\draw[gray!45,line width=0.25pt] (1.700,-0.680) rectangle (2.040,-0.340);
\draw[gray!45,line width=0.25pt] (1.700,-0.340) rectangle (2.040,0.000);
\draw[gray!45,line width=0.25pt] (1.700,0.000) rectangle (2.040,0.340);
\draw[gray!45,line width=0.25pt] (1.700,0.340) rectangle (2.040,0.680);
\draw[gray!45,line width=0.25pt] (1.700,0.680) rectangle (2.040,1.020);
\draw[gray!45,line width=0.25pt] (2.040,-0.680) rectangle (2.380,-0.340);
\draw[gray!45,line width=0.25pt] (2.040,-0.340) rectangle (2.380,0.000);
\draw[gray!45,line width=0.25pt] (2.040,0.000) rectangle (2.380,0.340);
\draw[gray!45,line width=0.25pt] (2.040,0.340) rectangle (2.380,0.680);
\fill[gray!35] (-2.210,-0.510) ellipse (0.092 and 0.068);
\fill[gray!35] (-2.210,-0.170) ellipse (0.092 and 0.068);
\fill[gray!35] (-2.210,0.170) ellipse (0.092 and 0.068);
\fill[gray!35] (-2.210,0.510) ellipse (0.092 and 0.068);
\fill[gray!35] (-1.870,-0.850) ellipse (0.092 and 0.068);
\fill[gray!35] (-1.870,-0.510) ellipse (0.092 and 0.068);
\fill[gray!35] (-1.870,-0.170) ellipse (0.092 and 0.068);
\fill[gray!35] (-1.870,0.170) ellipse (0.092 and 0.068);
\fill[gray!35] (-1.870,0.510) ellipse (0.092 and 0.068);
\fill[gray!35] (-1.870,0.850) ellipse (0.092 and 0.068);
\fill[gray!35] (-1.530,-0.850) ellipse (0.092 and 0.068);
\fill[gray!35] (-1.530,-0.510) ellipse (0.092 and 0.068);
\fill[gray!35] (-1.530,-0.170) ellipse (0.092 and 0.068);
\fill[gray!35] (-1.530,0.170) ellipse (0.092 and 0.068);
\fill[gray!35] (-1.530,0.510) ellipse (0.092 and 0.068);
\fill[gray!35] (-1.530,0.850) ellipse (0.092 and 0.068);
\fill[gray!35] (-1.190,-1.190) ellipse (0.092 and 0.068);
\fill[gray!35] (-1.190,-0.850) ellipse (0.092 and 0.068);
\fill[gray!35] (-1.190,-0.510) ellipse (0.092 and 0.068);
\fill[gray!35] (-1.190,-0.170) ellipse (0.092 and 0.068);
\fill[gray!35] (-1.190,0.170) ellipse (0.092 and 0.068);
\fill[gray!35] (-1.190,0.510) ellipse (0.092 and 0.068);
\fill[gray!35] (-1.190,0.850) ellipse (0.092 and 0.068);
\fill[gray!35] (-1.190,1.190) ellipse (0.092 and 0.068);
\fill[gray!35] (-0.850,-1.190) ellipse (0.092 and 0.068);
\fill[gray!35] (-0.850,-0.850) ellipse (0.092 and 0.068);
\fill[gray!35] (-0.850,-0.510) ellipse (0.092 and 0.068);
\fill[gray!35] (-0.850,-0.170) ellipse (0.092 and 0.068);
\fill[gray!35] (-0.850,0.170) ellipse (0.092 and 0.068);
\fill[gray!35] (-0.850,0.510) ellipse (0.092 and 0.068);
\fill[gray!35] (-0.850,0.850) ellipse (0.092 and 0.068);
\fill[gray!35] (-0.850,1.190) ellipse (0.092 and 0.068);
\fill[gray!35] (-0.510,-1.190) ellipse (0.092 and 0.068);
\fill[gray!35] (-0.510,-0.850) ellipse (0.092 and 0.068);
\fill[gray!35] (-0.510,-0.510) ellipse (0.092 and 0.068);
\fill[gray!35] (-0.510,-0.170) ellipse (0.092 and 0.068);
\fill[gray!35] (-0.510,0.170) ellipse (0.092 and 0.068);
\fill[gray!35] (-0.510,0.510) ellipse (0.092 and 0.068);
\fill[gray!35] (-0.510,0.850) ellipse (0.092 and 0.068);
\fill[gray!35] (-0.510,1.190) ellipse (0.092 and 0.068);
\fill[gray!35] (-0.170,-1.530) ellipse (0.092 and 0.068);
\fill[gray!35] (-0.170,-1.190) ellipse (0.092 and 0.068);
\fill[gray!35] (-0.170,-0.850) ellipse (0.092 and 0.068);
\fill[gray!35] (-0.170,-0.510) ellipse (0.092 and 0.068);
\fill[gray!35] (-0.170,-0.170) ellipse (0.092 and 0.068);
\fill[gray!35] (-0.170,0.170) ellipse (0.092 and 0.068);
\fill[gray!35] (-0.170,0.510) ellipse (0.092 and 0.068);
\fill[gray!35] (-0.170,0.850) ellipse (0.092 and 0.068);
\fill[gray!35] (-0.170,1.190) ellipse (0.092 and 0.068);
\fill[gray!35] (-0.170,1.530) ellipse (0.092 and 0.068);
\fill[gray!35] (0.170,-1.530) ellipse (0.092 and 0.068);
\fill[gray!35] (0.170,-1.190) ellipse (0.092 and 0.068);
\fill[gray!35] (0.170,-0.850) ellipse (0.092 and 0.068);
\fill[gray!35] (0.170,-0.510) ellipse (0.092 and 0.068);
\fill[gray!35] (0.170,-0.170) ellipse (0.092 and 0.068);
\fill[gray!35] (0.170,0.170) ellipse (0.092 and 0.068);
\fill[gray!35] (0.170,0.510) ellipse (0.092 and 0.068);
\fill[gray!35] (0.170,0.850) ellipse (0.092 and 0.068);
\fill[gray!35] (0.170,1.190) ellipse (0.092 and 0.068);
\fill[gray!35] (0.170,1.530) ellipse (0.092 and 0.068);
\fill[gray!35] (0.510,-1.190) ellipse (0.092 and 0.068);
\fill[gray!35] (0.510,-0.850) ellipse (0.092 and 0.068);
\fill[gray!35] (0.510,-0.510) ellipse (0.092 and 0.068);
\fill[gray!35] (0.510,-0.170) ellipse (0.092 and 0.068);
\fill[gray!35] (0.510,0.170) ellipse (0.092 and 0.068);
\fill[gray!35] (0.510,0.510) ellipse (0.092 and 0.068);
\fill[gray!35] (0.510,0.850) ellipse (0.092 and 0.068);
\fill[gray!35] (0.510,1.190) ellipse (0.092 and 0.068);
\fill[gray!35] (0.850,-1.190) ellipse (0.092 and 0.068);
\fill[gray!35] (0.850,-0.850) ellipse (0.092 and 0.068);
\fill[gray!35] (0.850,-0.510) ellipse (0.092 and 0.068);
\fill[gray!35] (0.850,-0.170) ellipse (0.092 and 0.068);
\fill[gray!35] (0.850,0.170) ellipse (0.092 and 0.068);
\fill[gray!35] (0.850,0.510) ellipse (0.092 and 0.068);
\fill[gray!35] (0.850,0.850) ellipse (0.092 and 0.068);
\fill[gray!35] (0.850,1.190) ellipse (0.092 and 0.068);
\fill[gray!35] (1.190,-1.190) ellipse (0.092 and 0.068);
\fill[gray!35] (1.190,-0.850) ellipse (0.092 and 0.068);
\fill[gray!35] (1.190,-0.510) ellipse (0.092 and 0.068);
\fill[gray!35] (1.190,-0.170) ellipse (0.092 and 0.068);
\fill[gray!35] (1.190,0.170) ellipse (0.092 and 0.068);
\fill[gray!35] (1.190,0.510) ellipse (0.092 and 0.068);
\fill[gray!35] (1.190,0.850) ellipse (0.092 and 0.068);
\fill[gray!35] (1.190,1.190) ellipse (0.092 and 0.068);
\fill[gray!35] (1.530,-0.850) ellipse (0.092 and 0.068);
\fill[gray!35] (1.530,-0.510) ellipse (0.092 and 0.068);
\fill[gray!35] (1.530,-0.170) ellipse (0.092 and 0.068);
\fill[gray!35] (1.530,0.170) ellipse (0.092 and 0.068);
\fill[gray!35] (1.530,0.510) ellipse (0.092 and 0.068);
\fill[gray!35] (1.530,0.850) ellipse (0.092 and 0.068);
\fill[gray!35] (1.870,-0.850) ellipse (0.092 and 0.068);
\fill[gray!35] (1.870,-0.510) ellipse (0.092 and 0.068);
\fill[gray!35] (1.870,-0.170) ellipse (0.092 and 0.068);
\fill[gray!35] (1.870,0.170) ellipse (0.092 and 0.068);
\fill[gray!35] (1.870,0.510) ellipse (0.092 and 0.068);
\fill[gray!35] (1.870,0.850) ellipse (0.092 and 0.068);
\fill[gray!35] (2.210,-0.510) ellipse (0.092 and 0.068);
\fill[gray!35] (2.210,-0.170) ellipse (0.092 and 0.068);
\fill[gray!35] (2.210,0.170) ellipse (0.092 and 0.068);
\fill[gray!35] (2.210,0.510) ellipse (0.092 and 0.068);
\node at (2.25,1.35) {\(\Omega\)};
\draw[-{Stealth[length=2mm]},line width=0.5pt] (3.15,0.2) node[right] {\(\varepsilon(k+Y)\)} -- (2.33,0.31);
\draw[-{Stealth[length=2mm]},line width=0.5pt] (3.15,-0.55) node[right] {\(\varepsilon(k+T)\)} -- (1.88,-0.51);
\end{scope}
\end{tikzpicture}
\caption{Left: the reference cell \(Y\), with the hole \(T\), the solid part \(Y^{*}\) and the normal \(\nu\) on \(\partial T\), which points into the hole. Right: the perforated domain \(\Omega_{\varepsilon}\); only the complete cells \(\varepsilon(k+Y)\), \(k\in K_{\varepsilon}\), are perforated, and the layer along \(\partial\Omega\) is unperforated.}
\label{fig:geometry}
\end{figure}

Let \(\Phi\) be an \(N\)-function with derivative \(\phi=\Phi'\), \(F(y,\xi)\) a periodic convex bulk energy with flux \(A=D_{\xi}F\), and \(\rho\) a periodic surface density; the precise assumptions are stated in Subsection~\ref{subsec:hypotheses}. The outward unit normal  to \(\Omega_{\varepsilon}\)
on \(\Gamma_{\varepsilon}\) is denoted by \(\nu_{\varepsilon}\) and is the rescaling of \(\nu\). For a prescribed surface modular mass \(\alpha>0\), the normalized nonlinear Steklov problem takes the formal form
\begin{equation}
\left\{
\begin{aligned}
-\operatorname{div}
A\left(\frac{x}{\varepsilon},\nabla u_{\varepsilon,\alpha}\right)
&=0
&&\text{in }\Omega_{\varepsilon},
\\
u_{\varepsilon,\alpha}&=0
&&\text{on }\partial\Omega,
\\
A\left(\frac{x}{\varepsilon},\nabla u_{\varepsilon,\alpha}\right)
\cdot\nu_{\varepsilon}
&=
\varepsilon\lambda_{1,\varepsilon}(\alpha)
\rho\left(\frac{x}{\varepsilon}\right)
\phi(\lvert u_{\varepsilon,\alpha}\rvert)
\operatorname{sgn}(u_{\varepsilon,\alpha})
&&\text{on }\Gamma_{\varepsilon},
\\
\varepsilon\int_{\Gamma_{\varepsilon}}
\rho\left(\frac{x}{\varepsilon}\right)
\Phi(\lvert u_{\varepsilon,\alpha}\rvert)\,d\sigma_{x}
&=\alpha,
\end{aligned}
\right.
\label{eq:intro-strong}
\end{equation}
where \(\operatorname{sgn}(t)=t/\lvert t\rvert\) for \(t\neq0\) and \(\operatorname{sgn}(0)=0\).
The associated bulk energy and surface modular are
\begin{align}
\mathcal{E}_{\varepsilon}(v)
&=
\int_{\Omega_{\varepsilon}}
F\left(\frac{x}{\varepsilon},\nabla v(x)\right)\,dx,
\label{eq:intro-energy}
\\
\mathcal{M}_{\varepsilon}(v)
&=
\varepsilon
\int_{\Gamma_{\varepsilon}}
\rho\left(\frac{x}{\varepsilon}\right)
\Phi\bigl(\lvert v(x)\rvert\bigr)\,d\sigma_{x},
\label{eq:intro-mass}
\end{align}
Since there are of order \(\varepsilon^{-N}\) holes, each with surface area
of order \(\varepsilon^{N-1}\), the total internal surface area is of order
\(\varepsilon^{-1}\). The factor \(\varepsilon\) makes the surface
normalization contribute at order one
\cite{AllaireDamlamianHornung1996,CioranescuDonato1988}.  We consider the constrained minimum
\begin{equation}
\mu_{1,\varepsilon}(\alpha)
=
\inf
\left\{
\mathcal{E}_{\varepsilon}(v):
v\in X_{\varepsilon},\
\mathcal{M}_{\varepsilon}(v)=\alpha
\right\},
\label{eq:intro-minimum}
\end{equation}
where
\begin{equation}
X_{\varepsilon}
=
\left\{
v\in W^{1}L^{\Phi}(\Omega_{\varepsilon}):
\gamma_{0}v=0\ \text{on }\partial\Omega
\right\}
\label{eq:intro-space}
\end{equation}
and \(\gamma_{0}\) denotes the trace on \(\partial\Omega\); traces in integrals
over \(\Gamma_{\varepsilon}\) are understood without writing \(\gamma_0\).
The Orlicz and Orlicz--Sobolev spaces used here are defined in
Subsection~\ref{subsec:orlicz}.

The variational formulation of \eqref{eq:intro-strong} for a pair
\(\bigl(\lambda_{1,\varepsilon}(\alpha),u_{\varepsilon,\alpha}\bigr)\in\mathbb{R}\times X_{\varepsilon}\)
reads as follows: for every \(v\in X_{\varepsilon}\),
\begin{equation}
\int_{\Omega_{\varepsilon}}
\sum_{i=1}^{N}
A_{i}\left(\frac{x}{\varepsilon},\nabla u_{\varepsilon,\alpha}\right)
D_{x_{i}}v\,dx
=
\varepsilon\lambda_{1,\varepsilon}(\alpha)
\int_{\Gamma_{\varepsilon}}
\rho\left(\frac{x}{\varepsilon}\right)
\phi\bigl(\lvert u_{\varepsilon,\alpha}\rvert\bigr)
\operatorname{sgn}(u_{\varepsilon,\alpha})
v\,d\sigma_{x}.
\label{eq:intro-weak}
\end{equation}

\begin{definition}
\label{def:normalized-eigencouple}
A normalized principal eigencouple of mass \(\alpha\) is a pair
\(\bigl(\lambda_{1,\varepsilon}(\alpha),u_{\varepsilon,\alpha}\bigr)\in\mathbb{R}\times X_{\varepsilon}\)
such that \(\mathcal{M}_{\varepsilon}(u_{\varepsilon,\alpha})=\alpha\),
\(\mathcal{E}_{\varepsilon}(u_{\varepsilon,\alpha})=\mu_{1,\varepsilon}(\alpha)\), and \eqref{eq:intro-weak} holds.  Here ``principal'' refers to the first constrained variational level, and \(u_{\varepsilon,\alpha}\) is called a normalized minimizer; the same terminology is used for the minimizers of the homogenized problem \eqref{eq:intro-mu0} below.
\end{definition}

By Theorem~\ref{thm:micro-existence}, the minimum in \eqref{eq:intro-minimum} is attained, every minimizer gives a normalized principal eigencouple, and a minimizer may be chosen nonnegative; by Remark~\ref{rem:multipliers}, the multiplier is determined by the minimizer.  For each sufficiently small \(\varepsilon\), we fix such a nonnegative eigencouple.

For the chosen normalized minimizer, testing \eqref{eq:intro-weak} by \(u_{\varepsilon,\alpha}\) gives
\begin{equation}
\lambda_{1,\varepsilon}(\alpha)
=\frac{\displaystyle\int_{\Omega_{\varepsilon}}
 A\!\left(\frac{x}{\varepsilon},\nabla u_{\varepsilon,\alpha}\right)\cdot\nabla u_{\varepsilon,\alpha}\,dx}
 {\displaystyle\varepsilon\int_{\Gamma_{\varepsilon}}\rho\!\left(\frac{x}{\varepsilon}\right)
 \phi(|u_{\varepsilon,\alpha}|)|u_{\varepsilon,\alpha}|\,d\sigma_x}.
\label{eq:intro-multiplier-ratio}
\end{equation}
The denominator is positive under the constraint of mass \(\alpha\). For a general \(N\)-function, the ratio differs from \(\mu_{1,\varepsilon}(\alpha)/\alpha\); the selected minimizer, its energy and its multiplier are therefore followed together.

The model can be interpreted as diffusion through a porous matrix with
nonlinear adsorption on its internal walls: the surface constraint fixes a
prescribed adsorbed mass, while the minimum measures the corresponding bulk
energy \cite{HornungJager1991}. Boundary-supported spectral
terms also occur in membrane and fluid problems \cite{GirouardPolterovich2017}.
Non-power constitutive laws arise, for example, in nonlinear conduction and
non-Newtonian flow \cite{GwiazdaSwierczewskaWroblewska2010,Zhikov1987};
related models include Carreau-type flows in porous media
\cite{AnguianoBonnivardSuarezGrau2025}.

\subsection{Related work}
\label{subsec:literature}

Spectral homogenization goes back to Kesavan \cite{Kesavan1979} for
oscillating coefficients and to Vanninathan \cite{Vanninathan1981} for
perforated domains, where the Dirichlet, Neumann and Steklov regimes are
distinguished by the scaling of the eigenvalue; error estimates for the Steklov
problem with smooth coefficients are given in \cite{Pastukhova1995}.  The
\(\varepsilon\)-scaled surface term is the one of the nonhomogeneous Neumann
problem in perforated domains studied by Cioranescu and Donato
\cite{CioranescuDonato1988}, in which surface data on the holes produce a
volume source in the limit.  Linear Steklov eigencouples in this regime were
homogenized by two-scale convergence, gradients included, in
\cite{Douanla2010}, and sign-changing surface densities lead to different
spectral regimes according to their cell average
\cite{ChiadoPiatPiatnitski2012,Douanla2013}.

Steklov problems in perforated domains have also been studied with holes
much smaller than the period.  Girouard, Henrot and Lagac\'e
\cite{GirouardHenrotLagace2021} consider balls in the critical regime where
the surface of each hole is comparable to the volume of its cell.  The holes
then leave the Laplacian unchanged, and the Steklov spectrum converges to
that of a weighted Laplace problem in the whole domain with a dynamical
boundary condition, which they use to relate Steklov and Neumann
eigenvalues.  In this regime, Khrabustovskyi and Taskinen
\cite{KhrabustovskyiTaskinen2026} prove norm-resolvent convergence and
estimates for the Hausdorff distance between the spectra, under general
assumptions on the shapes and locations of the holes.  Girouard and
Lagac\'e \cite{GirouardLagace2021} remove small geodesic balls from a
Riemannian manifold to realize weighted Laplace eigenvalues as limits of
Steklov eigenvalues, and obtain in this way manifolds with large Steklov
eigenvalues.  In our setting, as in
\cite{Douanla2010,Vanninathan1981}, the holes have the size of the period.
The factor \(\varepsilon\) turns the hole boundaries into a volume weight in
the limit, and the holes also modify the bulk operator through the cell
problem.

Baffico, Conca and Rajesh \cite{BafficoConcaRajesh2006} homogenize first
eigenvalues of monotone operators associated with even, positively
homogeneous convex energies; homogeneity makes the multiplier proportional to
the minimum.  Normalized eigenvalue problems in Orlicz--Sobolev spaces, where
the multiplier depends on the normalization level, are studied in
\cite{GossezManasevich2002,Salort2022}, homogeneous formulations based on the
Luxemburg norm in \cite{FernandezBonderSalortVivas2024}, and a constrained
optimization problem with an Orlicz boundary term in
\cite{DaSilvaSalortSilvaSpedaletti2019}.  Convergence rates for Steklov eigenvalues of the
\(p\)-Laplacian with rapidly oscillating weights on the boundary of a fixed
domain are obtained in \cite{SalortSteklov2022}.  The present paper brings together the
homogenization of Steklov problems in perforated domains and nonlinear
eigenvalue problems in Orlicz--Sobolev spaces: both the bulk energy and the
surface constraint have nonhomogeneous Orlicz growth, so that the constraint
is nonlinear and the multiplier depends on the mass.

The analysis is carried out in the framework of \(\Gamma\)-convergence
\cite{Braides2002,DalMaso1993} and of two-scale convergence
\cite{Allaire1992,Nguetseng1989}, its periodic-surface version
\cite{AllaireDamlamianHornung1996,NeussRadu1996}, and its Orlicz formulation
\cite{TachagoGargiuloNnangZappale2021,TachagoNnang2012,TachagoNnangTchindaZappale2024,TachagoNnangZappale2021},
together with uniform extension operators
\cite{AcerbiChiadoPiatDalMasoPercivale1992,CioranescuDonato1999}.  For
general facts on Orlicz and Orlicz--Sobolev spaces we refer to
\cite{DonaldsonTrudinger1971,Gossez1982,KrasnoselskiiRutickii1961,RaoRen1991},
and for Orlicz trace embeddings to \cite{CavaliereMihula2019,Cianchi2010}.

\subsection{Standing hypotheses}
\label{subsec:hypotheses}

In addition to \((H_0)\), we impose the following standing assumptions on the Orlicz growth, the bulk energy, and the surface density.

\medskip
\noindent
\((H_1)\)\quad
The function \(\Phi:[0,\infty)\to[0,\infty)\) is an \(N\)-function of the form
\begin{equation}
\Phi(t)=\int_{0}^{t}\phi(s)\,ds,
\label{eq:Phi-representation}
\end{equation}
where
\[
\phi\in C([0,\infty))\cap C^{1}((0,\infty)),
\qquad
\phi(0)=0,
\]
and \(\phi\) is strictly increasing on \((0,\infty)\).  Its complementary \(N\)-function is
\begin{equation}
\widetilde{\Phi}(t)
=
\sup_{s\geq0}\{st-\Phi(s)\},
\qquad t\geq0.
\label{eq:Phi-complement}
\end{equation}
There are constants \(1<p^{-}\leq p^{+}<\infty\) and \(\kappa^{+}\geq1\) such that, for every \(t>0\),
\begin{align}
p^{-}
&\leq
\frac{t\phi(t)}{\Phi(t)}
\leq p^{+},
\label{eq:Phi-indices}
\\
\frac{t\phi'(t)}{\phi(t)}
&\leq \kappa^{+}-1.
\label{eq:phi-indices}
\end{align}
Index conditions of this type are used, for instance, in \cite{Lieberman1991}. The bounds in \eqref{eq:Phi-indices} imply the global \(\Delta_2\)-conditions for both \(\Phi\) and \(\widetilde\Phi\), namely \(\Phi(2t)\leq C\Phi(t)\) and \(\widetilde\Phi(2t)\leq C\widetilde\Phi(t)\) for every \(t\geq0\): integrating the logarithmic derivative of \(\Phi\) gives its doubling property, while \(p^->1\) yields the complementary doubling property via convex duality. Condition \eqref{eq:phi-indices} controls the nonlinear surface quantity \(t\phi(t)\) in Lemma~\ref{lem:nonlinear-modular-continuity}.

\medskip
\noindent
\((H_2)\)\quad
The function
\[
F:\mathbb{R}^{N}\times\mathbb{R}^{N}\longrightarrow[0,\infty)
\]
is \(Y\)-periodic in its first variable, and its restriction to \(\overline{Y}\times\mathbb{R}^{N}\) is continuous.  For every \(y\), the function \(F(y,\cdot)\) is strictly convex and of class \(C^{1}\); we write
\[
A(y,\xi)=D_{\xi}F(y,\xi).
\]  There are constants \(0<c_{0}\leq c_{1}\) such that, for every \(y\in\mathbb{R}^{N}\) and every \(\xi\in\mathbb{R}^{N}\),
\begin{align}
F(y,-\xi)&=F(y,\xi),
\label{eq:F-even}
\\
c_{0}\Phi(\lvert\xi\rvert)
&\leq
F(y,\xi)
\leq
c_{1}\Phi(\lvert\xi\rvert).
\label{eq:F-growth}
\end{align}
In particular \(F(y,0)=0\).  The following properties of the flux are consequences of these assumptions: for every \(y,\xi,\eta\in\mathbb{R}^{N}\),
\begin{align}
\widetilde{\Phi}
\left(
\frac{\lvert A(y,\xi)\rvert}{c_{2}}
\right)
&\leq
\Phi(\lvert\xi\rvert),
\label{eq:A-growth}
\\
a_{0}\Phi(\lvert\xi\rvert)
&\leq
\sum_{i=1}^{N}A_{i}(y,\xi)\xi_{i}
\leq
a_{1}\Phi(\lvert\xi\rvert),
\label{eq:A-pairing}
\\
\sum_{i=1}^{N}
\bigl(A_{i}(y,\xi)-A_{i}(y,\eta)\bigr)
(\xi_{i}-\eta_{i})
&>0
\quad\text{whenever }\xi\neq\eta.
\label{eq:A-strict}
\end{align}
Here \(c_{2}=2^{p^{+}-1}c_{1}\), \(a_{0}=c_{0}\), and
\(a_{1}=2^{p^{+}}c_{1}\), and the restriction of \(A\) to \(\overline{Y}\times\mathbb{R}^{N}\) is continuous; these properties are derived in Subsection~\ref{subsec:orlicz}.  The evenness \eqref{eq:F-even} allows the selection of nonnegative minimizers.

\medskip
\noindent
\((H_3)\)\quad
The function \(\rho\in L^{\infty}(\partial T)\) is extended periodically to
\(\bigcup_{k\in\mathbb{Z}^{N}}(k+\partial T)\), and
\begin{equation}
0
\leq
\rho(y)
\leq
\rho_{+}
\quad\text{for \(d\sigma_{y}\)-almost every }y\in\partial T,
\qquad
\beta_{\rho}
=
\int_{\partial T}\rho(y)\,d\sigma_{y}>0.
\label{eq:rho-positive}
\end{equation}

\subsection{Main results and new ingredients}
\label{subsec:main-results}

We now introduce the limit problem and state the main results.

For \(\xi\in\mathbb{R}^{N}\), define
\begin{equation}
F_{\mathrm{hom}}(\xi)
=
\inf_{w\in W_{\#}^{1}L^{\Phi}(Y^{*})}
\int_{Y^{*}}
F\bigl(y,\xi+\nabla_{y}w(y)\bigr)\,dy.
\label{eq:intro-Fhom}
\end{equation}
This minimization is the cell problem; its unique minimizer (Proposition~\ref{prop:cell-minimizer}) is denoted by \(w_{\xi}\), and
\begin{equation}
A_{\mathrm{hom},i}(\xi)
=
\int_{Y^{*}}
A_{i}\bigl(y,\xi+\nabla_{y}w_{\xi}(y)\bigr)\,dy,
\qquad i=1,\ldots,N.
\label{eq:intro-Ahom}
\end{equation}
For \(v\in W_{0}^{1}L^{\Phi}(\Omega)\), put
\begin{align}
\mathcal{E}_{0}(v)
&=
\int_{\Omega}F_{\mathrm{hom}}(\nabla v)\,dx,
\label{eq:intro-E0}
\\
\mathcal{M}_{0}(v)
&=
\beta_{\rho}\int_{\Omega}\Phi(\lvert v\rvert)\,dx,
\label{eq:intro-M0}
\\
\mu_{1,0}(\alpha)
&=
\inf
\left\{
\mathcal{E}_{0}(v):
v\in W_{0}^{1}L^{\Phi}(\Omega),\
\mathcal{M}_{0}(v)=\alpha
\right\}.
\label{eq:intro-mu0}
\end{align}

\begin{maintheorem}
\label{thm:intro-main}
Assume \((H_0)\)--\((H_3)\), and fix \(\alpha>0\). For every sufficiently small \(\varepsilon\), let
\[
\bigl(
\lambda_{1,\varepsilon}(\alpha),
u_{\varepsilon,\alpha}
\bigr)
\]
be an arbitrary selection of nonnegative normalized principal eigencouples (Definition~\ref{def:normalized-eigencouple}), with the multiplier attached to the selected minimizer. Then the following hold.
\begin{enumerate}[label=\textup{(\roman*)},leftmargin=*,itemsep=2pt]
\item The constrained minima converge:
\begin{equation}
\mu_{1,\varepsilon}(\alpha)
\longrightarrow
\mu_{1,0}(\alpha).
\label{eq:intro-mu-convergence}
\end{equation}
\item From every sequence \(\varepsilon\to0\), one can extract a subsequence, not relabelled, and find a nonnegative minimizer
\[
u_{0,\alpha}\in W_{0}^{1}L^{\Phi}(\Omega)
\]
of \eqref{eq:intro-mu0}, a function
\[
u_{1,\alpha}
\in
L^{1}\left(
\Omega;
W_{\#}^{1}L^{\Phi}(Y^{*})
\right),
\qquad
\nabla_{y}u_{1,\alpha}
\in
L^{\Phi}(\Omega\times Y^{*})^{N},
\]
and a number \(\lambda_{1,0}(\alpha)>0\), associated with this limit minimizer, such that
\begin{align}
P_{\varepsilon}u_{\varepsilon,\alpha}
&\rightharpoonup
u_{0,\alpha}
&&\text{in }W_{0}^{1}L^{\Phi}(\Omega),
\label{eq:intro-weak-u}
\\
P_{\varepsilon}u_{\varepsilon,\alpha}
&\longrightarrow
u_{0,\alpha}
&&\text{in }L^{\Phi}(\Omega),
\label{eq:intro-strong-u}
\\
\chi_{\varepsilon}\nabla u_{\varepsilon,\alpha}
&\overset{2s}{\rightharpoonup}
\chi_{Y^{*}}
\bigl(\nabla u_{0,\alpha}+\nabla_{y}u_{1,\alpha}\bigr)
&&\text{weakly two-scale in }L^{\Phi}(\Omega)^{N},
\label{eq:intro-gradient-convergence}
\\
\chi_{\varepsilon}
A\left(
\frac{x}{\varepsilon},\nabla u_{\varepsilon,\alpha}
\right)
&\overset{2s}{\rightharpoonup}
\chi_{Y^{*}}
A\bigl(y,\nabla u_{0,\alpha}+\nabla_y u_{1,\alpha}\bigr)
&&\text{weakly two-scale in }L^{\widetilde{\Phi}}(\Omega)^{N},
\label{eq:intro-flux-convergence}
\\
\lambda_{1,\varepsilon}(\alpha)
&\longrightarrow
\lambda_{1,0}(\alpha).
\label{eq:intro-lambda-convergence}
\end{align}
Here \(P_{\varepsilon}\) is the uniform extension operator of
Proposition~\ref{prop:uniform-extension}, \(\chi_{\varepsilon}\) is the
characteristic function of \(\Omega_{\varepsilon}\) in \(\Omega\), and
\(\chi_{Y^{*}}\) is the characteristic function of \(Y^{*}\), extended
\(Y\)-periodically.  The two-scale convergences in
\eqref{eq:intro-gradient-convergence}--\eqref{eq:intro-flux-convergence}
are understood componentwise in the sense of
Definition~\ref{def:weak-two-scale}, with \(\Phi\) and \(\widetilde{\Phi}\)
interchanged in \eqref{eq:intro-flux-convergence}.  Moreover, the corrector is given by the cell problem:
\begin{equation}
u_{1,\alpha}(x,y)
=
w_{\nabla u_{0,\alpha}(x)}(y)
\quad\text{for almost every }(x,y)\in\Omega\times Y^{*}.
\label{eq:intro-cell-corrector}
\end{equation}
\item The limit minimizer satisfies the normalization
\begin{equation}
\beta_{\rho}
\int_{\Omega}\Phi(\lvert u_{0,\alpha}\rvert)\,dx
=
\alpha,
\label{eq:intro-limit-normalization}
\end{equation}
and the limit eigencouple \(\bigl(\lambda_{1,0}(\alpha),u_{0,\alpha}\bigr)\) satisfies, for every \(v\in W_{0}^{1}L^{\Phi}(\Omega)\),
\begin{equation}
\int_{\Omega}
A_{\mathrm{hom}}\bigl(\nabla u_{0,\alpha}\bigr)\cdot
\nabla v\,dx
=
\lambda_{1,0}(\alpha)\beta_{\rho}
\int_{\Omega}
\phi(\lvert u_{0,\alpha}\rvert)
\operatorname{sgn}(u_{0,\alpha})v\,dx.
\label{eq:intro-limit-equation}
\end{equation}
\end{enumerate}
\end{maintheorem}

\begin{maintheorem}
\label{thm:intro-full}
Under the assumptions of Theorem~\ref{thm:intro-main}, if \eqref{eq:intro-mu0} has a unique nonnegative minimizer, then all convergences of Theorem~\ref{thm:intro-main} hold for the whole family, without extraction.
\end{maintheorem}

In the linear case \(F(y,\xi)=\lvert\xi\rvert^{2}/2\),
\(\Phi(t)=t^{2}/2\) and \(\rho\equiv1\), Theorem~\ref{thm:intro-main} gives
a linear Dirichlet eigenvalue problem for the homogenized operator of the
perforated medium, with the constant volume weight \(\lvert\partial T\rvert\),
and the multiplier equals the minimum divided by \(\alpha\), consistently
with \cite{Douanla2010,Vanninathan1981}.

The convergence of the constrained minima follows the scheme of
\(\Gamma\)-convergence: a two-scale lower
bound and a recovery sequence which preserves the surface constraint
exactly.  The convergence of the multipliers and the identification of the flux
require further arguments.  For these, the proof compares each microscopic
trace with the average of the function over its cell, in the scaled surface
modular (Proposition~\ref{prop:nonlinear-surface-compactness}); this extends
to the Orlicz setting a comparison which, in Sobolev spaces, underlies the
boundary unfolding operator of the periodic unfolding method in perforated
domains \cite{CioranescuDonatoZaki2006}.  Combined with a Nemytskii lemma on the
varying surface measures, it recovers the nonlinear surface terms and their
derivative from strong volume convergence.  The resulting limit
of the flux--gradient pairing, obtained from the variational equation
\eqref{eq:intro-weak}, then enables the monotonicity argument that
identifies the effective flux and the multiplier.

\subsection{Organization of the paper}
\label{subsec:organization}

Section~\ref{sec:preliminaries} develops the Orlicz functional setting and
two-scale convergence, then establishes cell inequalities, trace bounds and the
uniform extension (Subsection~\ref{subsec:cell-inequalities}), the nonlinear
surface compactness result (Subsection~\ref{subsec:surface-compactness}) and
the compactness of gradients in perforated domains
(Subsection~\ref{subsec:perforated-compactness}).
Section~\ref{sec:homogenization} establishes microscopic solvability and
uniform estimates, constructs the effective cell law and the limit problem,
builds recovery sequences (Subsection~\ref{subsec:recovery}), proves the
convergence of normalized eigencouples, and closes with three benchmark models.  Extensions to a different surface \(N\)-function, sign-changing densities with zero cell average, and higher normalized eigencouples are natural directions for future work.

\section{Functional framework and two-scale tools}
\label{sec:preliminaries}

\subsection{Orlicz functional setting}
\label{subsec:orlicz}

We denote by \(E\) a fixed sequence of numbers \(\varepsilon\in(0,1)\) with \(K_{\varepsilon}\neq\varnothing\) and \(\varepsilon\to0\) (a fundamental sequence).  Families indexed by \(\varepsilon\), suprema over \(\varepsilon\), limits as \(\varepsilon\to0\), two-scale convergence and extractions of subsequences refer to \(\varepsilon\in E\).  Since \(E\) is arbitrary, a convergence proved for every \(E\) holds as \(\varepsilon\to0\) through \((0,1)\).  Finally, \(\mathbb{N}=\{1,2,\ldots\}\), \(\# K_{\varepsilon}\) is the cardinality of \(K_{\varepsilon}\), and \(\lvert\cdot\rvert\) denotes Lebesgue or surface measure.

We use the framework and notation of \cite[Section~2]{TachagoNnang2012} and
\cite[Subsection~2.1]{TachagoNnangZappale2021}; for standard facts on \(N\)-functions and Orlicz spaces we refer to \cite[Chapters~2--3]{RaoRen1991} and \cite[Chapter~II]{KrasnoselskiiRutickii1961}.  Under \((H_1)\), Young's equality \(\widetilde{\Phi}(\phi(s))=s\phi(s)-\Phi(s)\) and \eqref{eq:Phi-indices} give, for every \(s\geq0\),
\begin{equation}
\widetilde{\Phi}(\phi(s))
\leq
(p^{+}-1)\Phi(s).
\label{eq:phi-complement-bound}
\end{equation}
Integrating \(p^{-}/r\leq\phi(r)/\Phi(r)\leq p^{+}/r\) between \(t\) and \(st\) gives, for every \(t\geq0\),
\begin{align}
s^{p^{-}}\Phi(t)
&\leq
\Phi(st)
\leq
s^{p^{+}}\Phi(t)
&&\text{if }s\geq1,
\label{eq:scaling-large}
\\
s^{p^{+}}\Phi(t)
&\leq
\Phi(st)
\leq
s^{p^{-}}\Phi(t)
&&\text{if }0<s\leq1.
\label{eq:scaling-small}
\end{align}
Consequently, for every \(c>0\) there is \(C(c)>0\) such that, for every \(t\geq0\),
\begin{equation}
\Phi(ct)\leq C(c)\Phi(t),
\label{eq:fixed-scaling}
\end{equation}
and, for every \(d\in\mathbb{N}\), since \(\lvert\xi_{i}\rvert\leq\lvert\xi\rvert\leq d\max_{1\leq i\leq d}\lvert\xi_{i}\rvert\), there are \(c_{d},C_{d}>0\) such that, for every \(\xi\in\mathbb{R}^{d}\),
\begin{equation}
c_{d}\sum_{i=1}^{d}\Phi(\lvert\xi_{i}\rvert)
\leq
\Phi(\lvert\xi\rvert)
\leq
C_{d}\sum_{i=1}^{d}\Phi(\lvert\xi_{i}\rvert).
\label{eq:vector-component-modular-equivalence}
\end{equation}

The flux estimates \eqref{eq:A-growth}--\eqref{eq:A-strict} follow
from \((H_2)\). Indeed, convexity gives
\(A(y,\xi)\cdot\eta\leq F(y,\xi+\eta)-F(y,\xi)\) for every \(\eta\).
Taking \(\eta=-\xi\) and \(\eta=\xi\), and using
\eqref{eq:scaling-large}, yields
\[
c_{0}\Phi(|\xi|)\leq F(y,\xi)
\leq A(y,\xi)\cdot\xi
\leq F(y,2\xi)\leq 2^{p^{+}}c_{1}\Phi(|\xi|).
\]
For \(A(y,\xi)\ne0\), choose
\(\eta=s A(y,\xi)/|A(y,\xi)|\), \(s\geq0\); then
\[
s|A(y,\xi)|\leq c_{1}\Phi(|\xi|+s)
\leq c_{2}\bigl(\Phi(|\xi|)+\Phi(s)\bigr).
\]
The definition of \(\widetilde\Phi\) gives \eqref{eq:A-growth}.
Strict convexity of \(F(y,\cdot)\) gives \eqref{eq:A-strict}.
If \((y_n,\xi_n)\to(y,\xi)\) in \(\overline{Y}\times\mathbb{R}^{N}\), the convex functions \(F(y_n,\cdot)\) converge locally uniformly to the differentiable function \(F(y,\cdot)\), so \(A(y_n,\xi_n)\to A(y,\xi)\) \cite[Section~25]{Rockafellar1970}; thus \(A\) is continuous on \(\overline{Y}\times\mathbb{R}^{N}\).

For a finite measure space \((Q,\mathfrak{m})\), \(L^{\Phi}(Q,\mathfrak{m})\) is the Orlicz space with Luxemburg norm
\(\lVert v\rVert_{\Phi,Q,\mathfrak{m}}=\inf\{a>0:\int_{Q}\Phi(\lvert v\rvert/a)\,d\mathfrak{m}\leq1\}\); the measure is omitted when it is Lebesgue measure.  Since \(\Phi,\widetilde{\Phi}\in\Delta_{2}\), the space is separable and reflexive, \(\bigl(L^{\widetilde{\Phi}}(Q,\mathfrak{m})\bigr)'=L^{\Phi}(Q,\mathfrak{m})\), and, for every \(v\in L^{\Phi}(Q,\mathfrak{m})\) and every \(w\in L^{\widetilde{\Phi}}(Q,\mathfrak{m})\),
\begin{equation}
\int_{Q}\lvert vw\rvert\,d\mathfrak{m}
\leq
2\lVert v\rVert_{\Phi,Q,\mathfrak{m}}
\lVert w\rVert_{\widetilde{\Phi},Q,\mathfrak{m}}.
\label{eq:Orlicz-Holder}
\end{equation}
Inserting \(v/\lVert v\rVert_{\Phi,Q,\mathfrak{m}}\) in \eqref{eq:scaling-large}--\eqref{eq:scaling-small} gives, for every \(v\in L^{\Phi}(Q,\mathfrak{m})\),
\begin{equation}
\min\bigl\{\lVert v\rVert^{p^{-}},\lVert v\rVert^{p^{+}}\bigr\}
\leq
\int_{Q}\Phi(\lvert v\rvert)\,d\mathfrak{m}
\leq
\max\bigl\{\lVert v\rVert^{p^{-}},\lVert v\rVert^{p^{+}}\bigr\},
\qquad \lVert v\rVert=\lVert v\rVert_{\Phi,Q,\mathfrak{m}},
\label{eq:modular-small-norm}
\end{equation}
so that modular and norm convergence coincide.  For an open set \(Q\subset\mathbb{R}^{d}\), \(W^{1}L^{\Phi}(Q)\) is the space of \(v\in L^{\Phi}(Q)\) with \(D_{x_{i}}v\in L^{\Phi}(Q)\), \(i=1,\ldots,d\), normed by
\(\lVert v\rVert_{\Phi,Q}+\sum_{i=1}^{d}\lVert D_{x_{i}}v\rVert_{\Phi,Q}\), and \(W_{0}^{1}L^{\Phi}(Q)\) is the closure of \(C_{c}^{\infty}(Q)\); \(W^{1}L^{\Phi}_{\mathrm{loc}}(\mathbb{R}^{d})\) consists of the functions whose restrictions to every bounded open set \(Q\) belong to \(W^{1}L^{\Phi}(Q)\).  Smooth functions are dense in these spaces under \(\Delta_{2}\) \cite{Gossez1982}.  If \(Q\) is bounded with diameter \(d_{Q}\), then, for every \(v\in W_{0}^{1}L^{\Phi}(Q)\),
\begin{equation}
\int_{Q}\Phi(\lvert v\rvert)\,dx
\leq
\int_{Q}\Phi\bigl(d_{Q}\lvert\nabla v\rvert\bigr)\,dx.
\label{eq:Orlicz-Poincare}
\end{equation}
Indeed, for \(v\in C_{c}^{\infty}(Q)\) and \(x=(x_{1},x')\), the function \(v(\cdot,x')\) is supported in an interval \(I(x')\) of length at most \(d_{Q}\); Jensen's inequality gives
\(\Phi(\lvert v(x)\rvert)\leq d_{Q}^{-1}\int_{I(x')}\Phi(d_{Q}\lvert D_{x_{1}}v(t,x')\rvert)\,dt\), and integration in \(x_{1}\) and \(x'\) proves \eqref{eq:Orlicz-Poincare}.  By \eqref{eq:modular-small-norm}, it implies \(\lVert v\rVert_{\Phi,Q}\leq C_{Q}\lVert\nabla v\rVert_{\Phi,Q}\).  For every bounded Lipschitz domain \(D\), the embedding \(W^{1}L^{\Phi}(D)\hookrightarrow L^{\Phi}(D)\) is compact \cite{DonaldsonTrudinger1971,RaoRen1991}:
\begin{equation}
W^{1}L^{\Phi}(\Omega)
\hookrightarrow\hookrightarrow
L^{\Phi}(\Omega).
\label{eq:compact-volume-embedding}
\end{equation}

\(C_{\mathrm{per}}(Y)\) and \(C_{\mathrm{per}}^{\infty}(Y)\) denote the continuous and smooth \(Y\)-periodic functions on \(\mathbb{R}^{N}\).  The space \(W_{\mathrm{per}}^{1}L^{\Phi}(Y^{*})\) consists of the functions of \(W^{1}L^{\Phi}(Y^{*})\) whose traces coincide on each pair of opposite faces of \(Y\), that is, which extend periodically to \(\mathbb R^N\setminus\bigcup_{k\in\mathbb Z^N}(k+\overline T)\).  Following \cite{TachagoNnangZappale2021},
\begin{equation}
W_{\#}^{1}L^{\Phi}(Y^{*})
=
\left\{
w\in W_{\mathrm{per}}^{1}L^{\Phi}(Y^{*}):
\int_{Y^{*}}w(y)\,dy=0
\right\},
\label{eq:periodic-Orlicz-space}
\end{equation}
normed by \(\sum_{i=1}^{N}\lVert D_{y_{i}}w\rVert_{\Phi,Y^{*}}\); by Lemma~\ref{lem:Lipschitz-Poincare} below and Jensen's inequality, this is a norm.  The spaces \(W_{\mathrm{per}}^{1}L^{\Phi}(Y)\) and \(W_{\#}^{1}L^{\Phi}(Y)\) are defined in the same way.  For a separable Banach space \(X\), \(L^{1}(\Omega;X)\) is the Bochner space of strongly measurable functions \(u:\Omega\to X\) with \(\lVert u(\cdot)\rVert_{X}\in L^{1}(\Omega)\); likewise, \(L^{\widetilde{\Phi}}(\Omega;C_{\mathrm{per}}(Y))\) is the space of strongly measurable \(\psi:\Omega\to C_{\mathrm{per}}(Y)\) with \(x\mapsto\sup_{y\in Y}\lvert\psi(x,y)\rvert\) in \(L^{\widetilde{\Phi}}(\Omega)\) \cite{TachagoNnang2012}.

The boundary modular uses the same \(N\)-function as the volume space.  The following small-parameter estimate gives the trace compactness needed at fixed \(\varepsilon\); the hypothesis \(p^{-}>1\) supplies it through the factor \(\delta^{p^{-}-1}\).

\begin{proposition}
\label{prop:compact-Orlicz-trace}
Let \(D\subset\mathbb{R}^{N}\) be a bounded Lipschitz domain and let
\(\Sigma\subset\partial D\) be a compact \(C^{1,1}\) boundary portion which
admits a one-sided tubular neighbourhood contained in \(D\).  Under
\((H_1)\), there are constants
\(\delta_{0}>0\) and \(C>0\), depending only on the local geometry of
\(\Sigma\) in \(D\) and on \(\Phi\), such that, for every
\(0<\delta<\delta_{0}\) and every \(w\in W^{1}L^{\Phi}(D)\), writing
\(D_{\delta}=D_{\delta}(\Sigma)=\{z+s\,n_{D,\mathrm{in}}(z):z\in\Sigma,\ 0<s<\delta\}\) for the inner collar of width \(\delta\) along \(\Sigma\), where \(n_{D,\mathrm{in}}\) is the inward unit normal to \(D\) along \(\Sigma\), one has
\begin{equation}
\int_{\Sigma}\Phi(\lvert\gamma_{0}w\rvert)\,d\sigma
\leq \frac{C}{\delta}\int_{D_{\delta}}\Phi(C\lvert w\rvert)\,dx
+C\delta^{p^{-}-1}\int_{D_{\delta}}\Phi(\lvert\nabla w\rvert)\,dx.
\label{eq:small-parameter-trace}
\end{equation}
Consequently,
\begin{equation}
\gamma_{0}:W^{1}L^{\Phi}(D)\longrightarrow L^{\Phi}(\Sigma)
\quad\text{is continuous and compact.}
\label{eq:abstract-compact-trace}
\end{equation}
In particular,
\begin{equation}
\gamma_{0}:
W^{1}L^{\Phi}(Y^{*})
\longrightarrow
L^{\Phi}(\partial T)
\quad\text{is continuous and compact}.
\label{eq:trace-embedding}
\end{equation}
For each fixed \(\varepsilon>0\) with \(K_{\varepsilon}\neq\varnothing\),
the same conclusion holds from
\(W^{1}L^{\Phi}(\Omega_{\varepsilon})\) into
\(L^{\Phi}(\Gamma_{\varepsilon})\), the latter with the surface measure \(\sigma\).
\end{proposition}

\begin{proof}
We first prove \eqref{eq:small-parameter-trace} for smooth functions.  Since
\(\Sigma\) is \(C^{1,1}\), after decreasing \(\delta_{0}\) if necessary,
there is a one-sided tubular parametrization
\[
\Theta:\Sigma\times(0,\delta_{0})\longrightarrow D,
\qquad
\Theta(z,s)=z+s n_{D,\mathrm{in}}(z),
\]
where \(n_{D,\mathrm{in}}\) is the inward unit normal to \(D\) along \(\Sigma\).  The Jacobian of \(\Theta\) and
that of its inverse are bounded above and below by positive constants.

Fix \(z\in\Sigma\) and \(0<s<\delta_{0}\).  The fundamental theorem of
calculus along the inward normal gives
\[
\lvert w(z)\rvert
\leq
\lvert w(\Theta(z,s))\rvert
+
\int_{0}^{s}
\lvert\nabla w(\Theta(z,r))\rvert\,dr.
\]
By convexity of \(\Phi\), \eqref{eq:fixed-scaling}, and Jensen's inequality,
\begin{equation*}
\Phi(\lvert w(z)\rvert)
\leq
C\Phi(C\lvert w(\Theta(z,s))\rvert)
\quad+
\frac{C}{s}
\int_{0}^{s}
\Phi\bigl(Cs\lvert\nabla w(\Theta(z,r))\rvert\bigr)\,dr.
\end{equation*}
Average this inequality for \(s\in(\delta/2,\delta)\).  Since
\(s^{-1}\leq2/\delta\) on this interval and \(s\leq\delta\), Fubini's
theorem gives
\begin{equation}
\Phi(\lvert w(z)\rvert)
\leq
\frac{C}{\delta}
\int_{0}^{\delta}
\Phi(C\lvert w(\Theta(z,s))\rvert)\,ds
\quad+
\frac{C}{\delta}
\int_{0}^{\delta}
\Phi\bigl(C\delta\lvert\nabla w(\Theta(z,r))\rvert\bigr)\,dr.
\label{eq:pointwise-small-trace}
\end{equation}
Choose \(\delta_{0}\) still smaller so that \(C\delta_{0}\leq1\).  The
small-scale estimate \eqref{eq:scaling-small} yields
\[
\Phi(C\delta t)
\leq
(C\delta)^{p^{-}}\Phi(t)
\qquad
(0<\delta<\delta_{0},\ t\geq0).
\]
Integrating \eqref{eq:pointwise-small-trace} over \(\Sigma\), using the
uniform bounds for the tubular Jacobian, and then using
\eqref{eq:vector-component-modular-equivalence}, proves
\eqref{eq:small-parameter-trace} for smooth \(w\).  Density in
\(W^{1}L^{\Phi}(D)\), which is available under the global
\(\Delta_{2}\)-condition, extends the estimate and defines the continuous
trace operator.  Fixing one \(\delta\in(0,\delta_{0})\) in
\eqref{eq:small-parameter-trace} gives its continuity into
\(L^{\Phi}(\Sigma)\).

For compactness, let \(w_{n}\rightharpoonup w\) in \(W^{1}L^{\Phi}(D)\) and \(z_{n}=w_{n}-w\).  By \eqref{eq:compact-volume-embedding}, \eqref{eq:modular-small-norm} and \eqref{eq:fixed-scaling}, \(\int_{D}\Phi(C\lvert z_{n}\rvert)\,dx\to0\) for every \(C>0\), while \(M=\sup_{n\in\mathbb{N}}\int_{D}\Phi(\lvert\nabla z_{n}\rvert)\,dx<\infty\).  Applying \eqref{eq:small-parameter-trace} to \(z_{n}\) gives
\(\limsup_{n\to\infty}\int_{\Sigma}\Phi(\lvert\gamma_{0}z_{n}\rvert)\,d\sigma\leq CM\delta^{p^{-}-1}\) for every \(\delta\in(0,\delta_{0})\).  Since \(p^{-}>1\), letting \(\delta\downarrow0\) and using \eqref{eq:modular-small-norm} gives \(\gamma_{0}w_{n}\to\gamma_{0}w\) in \(L^{\Phi}(\Sigma)\).  For \(D=Y^{*}\), \(\Sigma=\partial T\), the inclusion \(\overline T\subset Y\) provides the collar, which gives \eqref{eq:trace-embedding}; for fixed \(\varepsilon\), \(\Gamma_{\varepsilon}\) is a finite disjoint union of scaled copies of \(\partial T\) inside \(\Omega\), and the same argument applies with constants depending on \(\varepsilon\).
\end{proof}

\begin{lemma}[Nemytskii lemma]
\label{lem:nonlinear-modular-continuity}
\textup{(a)} Let \((Q_n,\mathfrak{m}_n)\), \(n\in\mathbb N\), be finite measure spaces with
\(\sup_{n\in\mathbb N}\mathfrak{m}_n(Q_n)<\infty\), and let \(a_n,b_n\in L^\Phi(Q_n,\mathfrak{m}_n)\)
satisfy, for every \(C>0\),
\begin{align}
\int_{Q_n}\Phi(|a_n-b_n|)\,d\mathfrak{m}_n&\longrightarrow0,
\label{eq:nonlinear-difference}\\
\sup_{n\in\mathbb N}\int_{Q_n}
 [\Phi(C|a_n|)+\Phi(C|b_n|)]\,d\mathfrak{m}_n&<\infty.
\label{eq:nonlinear-boundedness}
\end{align}
Put \(H(t)=t\phi(t)\) for \(t\ge0\) and
\(g(t)=\phi(|t|)\operatorname{sgn}(t)\) for \(t\in\mathbb R\).
Then the differences
\[
\Phi(|a_n|)-\Phi(|b_n|),\qquad
H(|a_n|)-H(|b_n|),\qquad g(a_n)-g(b_n)
\]
converge to zero in \(L^1(Q_n,\mathfrak{m}_n)\).

\textup{(b)} Let \(D\subset\mathbb R^d\) have finite measure and let
\(\mathcal{B}:D\times\mathbb R^N\to\mathbb R^N\) be a Carath\'eodory mapping such that
\(\widetilde\Phi(|\mathcal{B}(z,\xi)|/c)\le C\Phi(|\xi|)\) for almost every \(z\)
and every \(\xi\). If \(V_n\to V\) in \(L^\Phi(D)^N\), then
\(\mathcal{B}(\cdot,V_n)\to \mathcal{B}(\cdot,V)\) in \(L^{\widetilde\Phi}(D)^N\).
Moreover, if \(\mathcal{J}:D\times\mathbb R^N\to\mathbb R\) is Carath\'eodory
and \(|\mathcal{J}(z,\xi)|\le C_{\mathcal{J}}\Phi(|\xi|)\), then
\(\mathcal{J}(\cdot,V_n)\to \mathcal{J}(\cdot,V)\) in \(L^1(D)\).
In particular, part \textup{(b)} applies to \(\mathcal{B}=A\) by \eqref{eq:A-growth}, and its last assertion to the integrand \(F\) under
\((H_2)\), with its spatial argument understood periodically.
\end{lemma}

\begin{proof}
(a) By \eqref{eq:phi-indices},
\(H'(t)\le\kappa^+\phi(t)\) and \(\phi\) is doubling. Thus, for
\(r,s\ge0\),
\[
|\Phi(r)-\Phi(s)|+|H(r)-H(s)|
 \le C|r-s|[\phi(r)+\phi(s)].
\]
The modular bound, \eqref{eq:phi-complement-bound} and
\eqref{eq:modular-small-norm} bound \(\phi(|a_n|)\) and \(\phi(|b_n|)\) uniformly in
\(L^{\widetilde\Phi}(Q_n,\mathfrak{m}_n)\), while
\(\|a_n-b_n\|_{\Phi,Q_n,\mathfrak{m}_n}\to0\).
H\"older's inequality \eqref{eq:Orlicz-Holder} proves the first two conclusions. For the third,
\(a_n-b_n\to0\) in measure on the varying measure spaces, and
continuity of \(g\), uniformly on
bounded intervals, implies convergence of the difference in measure.
To justify the passage on the varying measure spaces, set
\(h_n=g(a_n)-g(b_n)\). Convexity of \(\widetilde\Phi\),
\eqref{eq:phi-complement-bound}, and
\eqref{eq:nonlinear-boundedness} give
\[
\sup_{n\in\mathbb N}
\int_{Q_n}
\widetilde\Phi\!\left(\frac{|h_n|}{2}\right)
\,d\mathfrak m_n
\leq
\frac12\sup_{n\in\mathbb N}
\int_{Q_n}
\left[
\widetilde\Phi(\phi(|a_n|))
+\widetilde\Phi(\phi(|b_n|))
\right]d\mathfrak m_n
<\infty.
\]
Since \(\widetilde\Phi(t)/t\to\infty\) as \(t\to\infty\), this is a
de la Vall\'ee Poussin bound which makes \((h_n)_n\) uniformly integrable
with respect to \((\mathfrak m_n)_n\). To make the varying-space argument
explicit, set \(M=1+\sup_n\mathfrak m_n(Q_n)\). Given \(\eta>0\), uniform
integrability supplies \(\delta>0\) such that
\(\mathfrak m_n(E)<\delta\) implies
\(\int_E|h_n|\,d\mathfrak m_n<\eta/2\), uniformly in \(n\). With
\(\tau=\eta/(2M)\), convergence in measure gives
\(\mathfrak m_n(\{|h_n|>\tau\})<\delta\) for all sufficiently large \(n\),
and therefore
\[
\int_{Q_n}|h_n|\,d\mathfrak m_n
\leq \tau\mathfrak m_n(Q_n)
+\int_{\{|h_n|>\tau\}}|h_n|\,d\mathfrak m_n
<\eta.
\]

(b) Strong convergence in \(L^\Phi(D)^N\) implies convergence in measure
and uniform integrability of \(\{\Phi(|V_n|)\}_n\); indeed, by convexity
and \eqref{eq:fixed-scaling},
\[
\Phi(|V_n|)\le\tfrac12\Phi(2|V|)
 +\tfrac12\Phi(2|V_n-V|),
\qquad\int_D\Phi(2|V_n-V|)\,dz\longrightarrow0.
\]
Along any almost-everywhere convergent subsequence, continuity of \(\mathcal B\)
and convexity of \(\widetilde\Phi\) give
\[
\widetilde\Phi\!\left(
 \frac{|\mathcal{B}(z,V_n)-\mathcal{B}(z,V)|}{2c}\right)
 \le\frac C2\,[\Phi(|V_n|)+\Phi(|V|)].
\]
Vitali's theorem, followed by the \(\Delta_2\)-condition for
\(\widetilde\Phi\), implies strong convergence in \(L^{\widetilde\Phi}(D)^N\).
The bound on \(\mathcal{J}\) gives the analogous uniformly integrable majorants
\(|\mathcal{J}(z,V_n)|\le C_{\mathcal{J}}\Phi(|V_n|)\), hence its convergence in \(L^1(D)\).
Every subsequence admits such a further subsequence, which proves
convergence for the whole family.
\end{proof}

\begin{lemma}
\label{lem:smooth-periodic-density}
Set
\[
\mathcal{D}
=
C_{c}^{\infty}
\bigl(\Omega;C_{\mathrm{per}}^{\infty}(Y)^{N}\bigr).
\]
Then, for every
\(V\in L^{\Phi}(\Omega\times Y^{*})^{N}\), there exists
a sequence \((V_{n})_{n\in\mathbb{N}}\subset\mathcal{D}\) such that
\begin{equation}
V_{n}\longrightarrow V
\quad\text{in }L^{\Phi}(\Omega\times Y^{*})^{N}.
\label{eq:smooth-periodic-density}
\end{equation}
\end{lemma}

\begin{proof}
Extend \(V\) by zero to \(\Omega\times Y\) and periodically in \(y\), and denote the extension by \(\overline V\).  For cut-off functions \(\eta_{m}\in C_{c}^{\infty}(\Omega)\) with \(0\leq\eta_{m}\leq1\) and \(\eta_{m}\to1\) almost everywhere, \(\eta_{m}\overline V\to\overline V\) in \(L^{\Phi}(\Omega\times Y)^{N}\) by dominated convergence of the modular and \(\Delta_{2}\).  For fixed \(m\), mollification in \(x\) and periodic mollification in \(y\) approximate \(\eta_{m}\overline V\) in \(L^{\Phi}(\Omega\times Y)^{N}\) by elements of \(\mathcal{D}\), by translation continuity of the Luxemburg norm under \(\Delta_{2}\) \cite[Chapter~3]{RaoRen1991}.  A diagonal choice and restriction to \(\Omega\times Y^{*}\) give \eqref{eq:smooth-periodic-density}.
\end{proof}

\subsection{Two-scale convergence}
\label{subsec:two-scale}

We use the two-scale convergence of Nguetseng \cite{Nguetseng1989} in the Orlicz form and with the test spaces of \cite{TachagoNnang2012}.  Let \(\psi\) belong to
\(L^{\widetilde{\Phi}}(\Omega;C_{\mathrm{per}}(Y))\).  Its oscillating trace is
\begin{equation}
\psi^{\varepsilon}(x)
=
\psi\left(x,\frac{x}{\varepsilon}\right).
\label{eq:oscillating-test}
\end{equation}
This is well defined by continuous extension from
\(C(\overline{\Omega};C_{\mathrm{per}}(Y))\).  Indeed,
\[
\lvert\psi^{\varepsilon}(x)\rvert
\leq
\sup_{y\in Y}\lvert\psi(x,y)\rvert
\quad\text{for almost every }x\in\Omega.
\]

\begin{definition}[Weak two-scale convergence in \(L^{\Phi}\)]
\label{def:weak-two-scale}
A sequence \((v_{\varepsilon})_{\varepsilon\in E}\subset L^{\Phi}(\Omega)\) is said to converge weakly two-scale in \(L^{\Phi}(\Omega)\) to
$
v_{0}\in L^{\Phi}(\Omega\times Y)
$, extended \(Y\)-periodically in \(y\),
if, for every \(\psi\in L^{\widetilde{\Phi}}(\Omega;C_{\mathrm{per}}(Y))\),
\begin{equation}
\lim_{\varepsilon\to0}
\int_{\Omega}v_{\varepsilon}(x)
\psi\left(x,\frac{x}{\varepsilon}\right)\,dx
=
\int_{\Omega}\int_{Y}
v_{0}(x,y)\psi(x,y)\,dy\,dx.
\label{eq:weak-two-scale-definition}
\end{equation}
We write
 $ v_{\varepsilon}
\overset{2s}{\rightharpoonup}
v_{0}
\quad\text{in }L^{\Phi}(\Omega)$.
\end{definition}
For vector-valued sequences, weak two-scale convergence is understood componentwise.

We recall the Orlicz two-scale compactness result from \cite{TachagoNnang2012}.

\begin{theorem}
\label{thm:volume-two-scale-compactness}
Every bounded sequence \((v_{\varepsilon})_{\varepsilon\in E}\) in \(L^{\Phi}(\Omega)\) has a subsequence which converges weakly two-scale in the sense of Definition~\ref{def:weak-two-scale}.  The same holds with \(\widetilde{\Phi}\) in place of \(\Phi\).
\end{theorem}

On the hole boundaries we use the rescaled surface measure \(d\vartheta_{\varepsilon}=\varepsilon\,d\sigma_{\Gamma_{\varepsilon}}\).  The total masses
\(\vartheta_{\varepsilon}(\Gamma_{\varepsilon})=\varepsilon^{N}\,\#K_{\varepsilon}\,\lvert\partial T\rvert\leq\lvert\Omega\rvert\,\lvert\partial T\rvert\) are uniformly bounded, and the change of variables \(x=\varepsilon(k+y)\) gives, for every \(k\in K_{\varepsilon}\) and every integrable \(f\),
\begin{equation}
\int_{\varepsilon(k+\partial T)}f(x)\,d\vartheta_{\varepsilon}(x)
=\varepsilon^{N}\int_{\partial T}f\bigl(\varepsilon(k+y)\bigr)\,d\sigma_{y}.
\label{eq:one-cell-surface-Riemann}
\end{equation}

\subsection{Cell inequalities, trace bounds and uniform extension}
\label{subsec:cell-inequalities}

For later use, define
\begin{equation}
\Omega_{\varepsilon}^{\mathrm{cell}}
=
\bigcup_{k\in K_{\varepsilon}}\varepsilon(k+Y),
\qquad
R_{\varepsilon}
=
\Omega\setminus\overline{\Omega_{\varepsilon}^{\mathrm{cell}}};
\label{eq:boundary-layer}
\end{equation}
Since \(\Omega\) is Lipschitz, the boundary layer has thickness at most \(O(\varepsilon)\), and hence
\begin{equation}
\lvert R_{\varepsilon}\rvert\longrightarrow0
\quad\text{as }\varepsilon\longrightarrow0.
\label{eq:boundary-layer-small}
\end{equation}

We now prove the inequalities that make all later constants independent of
\(\varepsilon\).  The modular proofs below extend the Lipschitz extension results of
\cite{AcerbiChiadoPiatDalMasoPercivale1992,CioranescuDonato1999} to the
Orlicz setting and make explicit where \((H_1)\) enters.

\begin{lemma}
\label{lem:cell-average}
There is \(C_{Y}>0\) such that, for every
\(W\in W^{1}L^{\Phi}(Y)\),
\begin{equation}
\int_{Y}
\Phi
\left(
\left|
W(y)
-
\frac{1}{\lvert Y\rvert}\int_{Y}W(z)\,dz
\right|
\right)\,dy
\leq
C_{Y}
\int_{Y}
\sum_{i=1}^{N}
\Phi(C_{Y}\lvert D_{y_{i}}W\rvert)\,dy.
\label{eq:cell-average-inequality}
\end{equation}
\end{lemma}

\begin{proof}
For \(1\le j\le N\), let \(M_j\) average a function over its \(j\)-th coordinate on \((0,1)\), leaving the others fixed. These operators commute and are contractions for the Orlicz modular by Jensen's inequality. For \(W\in C^1(\overline Y)\), the one-dimensional fundamental theorem of calculus gives
\begin{align*}
|(I-M_j)W(y)|
&\le \int_0^1\int_{\min\{s,y_j\}}^{\max\{s,y_j\}}
|D_{y_j}W(y_1,\ldots,t,\ldots,y_N)|\,dt\,ds\\
&\le \int_0^1|D_{y_j}W(y_1,\ldots,t,\ldots,y_N)|\,dt
=M_j|D_{y_j}W|(y).
\end{align*}
Since \(|Y|=1\), the telescoping identity
\[
W-\int_Y W\,dy
=\sum_{j=1}^N M_1\cdots M_{j-1}(I-M_j)W
\]
therefore implies
\[
\left|W-\int_Y W\,dy\right|
\le\sum_{j=1}^N M_1\cdots M_j|D_{y_j}W|.
\]
Convexity and Jensen's inequality, first for the sum and then for the successive averages, yield
\begin{align*}
\int_Y\Phi\!\left(\left|W-\int_Y W\,dy\right|\right)dy
&\le \frac1N\sum_{j=1}^N\int_Y
 \Phi\!\left(NM_1\cdots M_j|D_{y_j}W|\right)dy\\
&\le\frac1N\sum_{j=1}^N\int_Y\Phi(N|D_{y_j}W|)\,dy.
\end{align*}
The fixed-scaling inequality \eqref{eq:fixed-scaling} gives \eqref{eq:cell-average-inequality}. Finally, smooth density in \(W^1L^\Phi(Y)\), valid under \(\Delta_2\) \cite{Gossez1982}, extends the estimate to every \(W\) in the stated space.
\end{proof}

\begin{lemma}
\label{lem:Lipschitz-Poincare}
Let \(D\subset\mathbb R^N\) be a bounded connected Lipschitz domain
and let \(B\Subset D\) be a ball. There is \(C_{D,B}>0\) such that,
for every \(w\in W^1L^\Phi(D)\), with \(w_B=|B|^{-1}\int_Bw\,dy\),
\begin{equation}
\int_D\Phi(|w-w_B|)\,dy
\le C_{D,B}\int_D\sum_{i=1}^N
 \Phi(C_{D,B}|D_{y_i}w|)\,dy.
\label{eq:Lipschitz-Poincare}
\end{equation}
\end{lemma}

\begin{proof}
Cover \(D\) by finitely many open sets \(V_{1},\ldots,V_{M}\subset D\), each bi-Lipschitz homeomorphic to \(Y\) through a map \(\Theta_{\ell}:Y\to V_{\ell}\), with \(\overline B\subset V_{1}\): interior cubes and, near \(\partial D\), intersections of \(D\) with Lipschitz coordinate cylinders. Join \(\ell\) and \(j\) whenever \(V_{\ell}\cap V_{j}\neq\varnothing\). Since the patches are open, every nonempty intersection has positive measure. Moreover, this overlap graph is connected: otherwise the unions of the patches belonging to two complementary collections of connected components would be disjoint nonempty relatively open subsets whose union is \(D\), contradicting the connectedness of \(D\). With \(m_{\ell}=\int_{Y}w(\Theta_{\ell}(t))\,dt\), Lemma~\ref{lem:cell-average} applied to \(w\circ\Theta_{\ell}\), the Jacobian bounds and \eqref{eq:fixed-scaling} give
\begin{equation}
\int_{V_\ell}\Phi(|w-m_\ell|)\,dy
 \le C\int_{V_\ell}\sum_{i=1}^N\Phi(C|D_{y_i}w|)\,dy,\qquad \ell=1,\ldots,M.
\label{eq:patch-Poincare}
\end{equation}
If \(\lvert V_{\ell}\cap V_{j}\rvert>0\), convexity and Jensen's inequality bound \(\Phi(|m_{\ell}-m_{j}|/2)\) by \((2\lvert V_{\ell}\cap V_{j}\rvert)^{-1}\) times the sum of the left-hand sides of \eqref{eq:patch-Poincare} for \(\ell\) and \(j\); similarly \(\Phi(|w_{B}-m_{1}|)\leq|B|^{-1}\int_{V_{1}}\Phi(|w-m_{1}|)\,dy\).  Following a path from \(1\) to \(\ell\) in the overlap graph bounds \(\Phi(|m_{\ell}-w_{B}|)\) by a constant times the right-hand side of \eqref{eq:Lipschitz-Poincare}.  The inequality \(\Phi(|w-w_{B}|)\leq\frac12\Phi(2|w-m_{\ell}|)+\frac12\Phi(2|m_{\ell}-w_{B}|)\) on each \(V_{\ell}\), summed over \(\ell=1,\ldots,M\), proves \eqref{eq:Lipschitz-Poincare}; the constants depend only on \(D\), \(B\) and the cover.  The norm inequality follows from \eqref{eq:modular-small-norm}.
\end{proof}

\begin{lemma}
\label{lem:cell-extension-trace}
There are a linear operator
\[
P:
W^{1}L^{\Phi}(Y^{*})
\longrightarrow
W^{1}L^{\Phi}(Y)
\]
and a constant \(C_{T}>0\) such that, for every \(w\in W^{1}L^{\Phi}(Y^{*})\),
\[
Pw=w\quad\text{in }Y^{*},
\qquad
P1=1,
\]
and
\begin{align}
\int_{Y}\Phi(\lvert Pw\rvert)\,dy
&\leq
C_{T}
\int_{Y^{*}}\Phi(C_{T}\lvert w\rvert)\,dy,
\label{eq:cell-extension-zero}
\\
\int_{Y}
\sum_{i=1}^{N}
\Phi(\lvert D_{y_{i}}Pw\rvert)\,dy
&\leq
C_{T}
\int_{Y^{*}}
\sum_{i=1}^{N}
\Phi(C_{T}\lvert D_{y_{i}}w\rvert)\,dy.
\label{eq:cell-extension-one}
\end{align}
Moreover, for every \(w\in W^{1}L^{\Phi}(Y^{*})\) and every \(W\in W^{1}L^{\Phi}(Y)\),
\begin{align}
\int_{\partial T}
\Phi(\lvert\gamma_{0}w\rvert)\,d\sigma_{y}
&\leq
C_{T}
\int_{Y^{*}}
\left[
\Phi(C_{T}\lvert w\rvert)
+
\sum_{i=1}^{N}
\Phi(C_{T}\lvert D_{y_{i}}w\rvert)
\right]\,dy,
\label{eq:cell-trace}
\\
\int_{\partial T}
\Phi
\left(
\left|
\gamma_{0}W
-
\int_{Y}W(z)\,dz
\right|
\right)\,d\sigma_{y}
&\leq
C_{T}
\int_{Y}
\sum_{i=1}^{N}
\Phi(C_{T}\lvert D_{y_{i}}W\rvert)\,dy,
\label{eq:cell-trace-Poincare}
\end{align}
where \(\gamma_{0}W\) is the trace of \(W\) on \(\partial T\).
\end{lemma}

\begin{proof}
Choose finitely many open cylinders \(U_{1},\ldots,U_{J}\), with closures in \(Y\), covering \(\partial T\), in which a rigid change of variables represents \(Y^{*}\cap U_{j}\) as \(\{(y',y_{N}):y_{N}>g_{j}(y')\}\cap U_{j}\) with \(g_{j}\) Lipschitz.  The reflections
\(R_{j}(y',y_{N})=\bigl(y',2g_{j}(y')-y_{N}\bigr)\)
are bi-Lipschitz.  Choose \(\theta_{j}\in C_{c}^{\infty}(U_{j})\), \(0\leq\theta_{j}\leq1\), with \(\sum_{j=1}^{J}\theta_{j}=1\) on a neighbourhood of \(\partial T\); after shrinking their supports, \(R_{j}\) maps \(T\cap\operatorname{supp}\theta_{j}\) into \(Y^{*}\cap U_{j}\).  For \(g\in C^{\infty}(\overline{Y^{*}})\) set
\[
P_{0}g=g\ \text{in }Y^{*},
\qquad
P_{0}g=\sum_{j=1}^{J}\theta_{j}\,(g\circ R_{j})\ \text{in }T.
\]
Since \(R_{j}\) is the identity on \(\partial T\) and
\(\sum_{j}\theta_{j}=1\) there, the two definitions have the same trace on
\(\partial T\), so \(P_{0}g\in W^{1}L^{\Phi}(Y)\). Changes of variables,
the chain rule, convexity, and \eqref{eq:fixed-scaling} give separately
\begin{align}
\int_{T}\Phi(\lvert P_{0}g\rvert)\,dy
&\leq
C\int_{Y^{*}}\Phi(C\lvert g\rvert)\,dy,
\label{eq:P0-zero}
\\
\int_{T}\sum_{i=1}^{N}
\Phi(\lvert D_{y_{i}}P_{0}g\rvert)\,dy
&\leq
C\int_{Y^{*}}
\left[
\Phi(C\lvert g\rvert)
+
\sum_{i=1}^{N}\Phi(C\lvert D_{y_{i}}g\rvert)
\right]dy.
\label{eq:P0-gradient}
\end{align}
Indeed, \eqref{eq:P0-zero} follows directly from the reflected pieces,
whereas the zero-order term on the right of \eqref{eq:P0-gradient} is
produced by the derivatives of the partition of unity.
Fix a ball \(B\Subset Y^{*}\) at positive distance from \(\overline{T}\cup\partial Y\), let \(c(w)\) be the average of \(w\) on \(B\), and define
\[
Pw=c(w)+P_{0}\bigl(w-c(w)\bigr).
\]
The operator \(P\) is linear, \(Pw=w\) in \(Y^{*}\), and \(P1=1\).
Since \(D(Pw)=D P_{0}(w-c(w))\), estimate
\eqref{eq:P0-gradient} and Lemma~\ref{lem:Lipschitz-Poincare} on the
connected Lipschitz set \(Y^{*}\) prove
\eqref{eq:cell-extension-one}. For the zero-order estimate, linearity gives
\[
Pw=P_{0}w+c(w)(1-P_{0}1).
\]
The function \(1-P_{0}1\) is bounded and supported in \(T\). Hence
\(P_0w=w\) on \(Y^*\), \eqref{eq:P0-zero}, convexity,
\eqref{eq:fixed-scaling}, and Jensen's inequality
\[
\Phi(\lvert c(w)\rvert)
\leq
\frac1{|B|}\int_B\Phi(\lvert w\rvert)\,dy
\]
give \eqref{eq:cell-extension-zero} without any gradient term.
Smooth density in \(W^{1}L^{\Phi}(Y^{*})\) under \(\Delta_{2}\)
\cite{Gossez1982} extends \(P\) to the whole space.

Estimate \eqref{eq:cell-trace} is \eqref{eq:small-parameter-trace} with \(D=Y^{*}\), \(\Sigma=\partial T\) and a fixed \(\delta\in(0,\delta_{0})\).  Finally, for \(W\in W^{1}L^{\Phi}(Y)\), set \(m=\int_{Y}W\,dz\) (recall \(\lvert Y\rvert=1\)) and apply \eqref{eq:cell-trace} to the restriction of \(W-m\) to \(Y^{*}\).  Lemma~\ref{lem:cell-average} bounds \(\int_{Y^{*}}\Phi(C\lvert W-m\rvert)\,dy\) by \(C\int_{Y}\sum_{i=1}^{N}\Phi(C\lvert D_{y_{i}}W\rvert)\,dy\), and the gradient term of \eqref{eq:cell-trace} is bounded directly.
  This proves \eqref{eq:cell-trace-Poincare}.
\end{proof}

\begin{proposition}
\label{prop:uniform-extension}
There are linear operators
\[
P_{\varepsilon}:
X_{\varepsilon}
\longrightarrow
W_{0}^{1}L^{\Phi}(\Omega)
\]
and a constant \(C>0\), independent of \(\varepsilon\), such that, for every \(v\in X_{\varepsilon}\),
\[
P_{\varepsilon}v=v
\quad\text{almost everywhere in }\Omega_{\varepsilon},
\]
and
\begin{align}
\int_{\Omega}
\Phi(\lvert P_{\varepsilon}v\rvert)\,dx
&\leq
C
\int_{\Omega_{\varepsilon}}
\Phi(C\lvert v\rvert)\,dx,
\label{eq:uniform-extension-zero}
\\
\int_{\Omega}
\sum_{i=1}^{N}
\Phi(\lvert D_{x_{i}}P_{\varepsilon}v\rvert)\,dx
&\leq
C
\int_{\Omega_{\varepsilon}}
\sum_{i=1}^{N}
\Phi(C\lvert D_{x_{i}}v\rvert)\,dx.
\label{eq:uniform-extension-one}
\end{align}
\end{proposition}

\begin{proof}
Fix \(v\in X_{\varepsilon}\).  For \(k\in K_{\varepsilon}\), define
\[
w_{\varepsilon,k}(y)
=
v\bigl(\varepsilon(k+y)\bigr),
\qquad y\in Y^{*}.
\]
On each complete cell, set
\begin{equation}
P_{\varepsilon}v(x)
=
Pw_{\varepsilon,k}
\left(
\frac{x}{\varepsilon}-k
\right)
\quad
\text{if }x\in\varepsilon(k+Y),\ k\in K_{\varepsilon},
\label{eq:extension-definition-cell}
\end{equation}
and set
\[
P_{\varepsilon}v(x)=v(x)
\quad\text{if }x\in R_{\varepsilon}.
\]
The reference extension in Lemma~\ref{lem:cell-extension-trace} equals \(w_{\varepsilon,k}\) near \(\partial Y\).  Hence the definitions have identical traces on every interface
\(\partial(\varepsilon(k+Y))\).  The patched function belongs to
\(W^{1}L^{\Phi}(\Omega)\).  The construction changes \(v\) only inside the holes, all of which stay away from \(\partial\Omega\); therefore
\[
\gamma_{0}P_{\varepsilon}v
=
\gamma_{0}v
=
0
\quad\text{on }\partial\Omega.
\]

For every \(k\in K_{\varepsilon}\), change variables in
\eqref{eq:cell-extension-zero}:
\begin{align*}
\int_{\varepsilon(k+Y)}
\Phi(\lvert P_{\varepsilon}v\rvert)\,dx
&=
\varepsilon^{N}
\int_{Y}
\Phi(\lvert Pw_{\varepsilon,k}\rvert)\,dy
\\
&\leq
C\varepsilon^{N}
\int_{Y^{*}}
\Phi(C\lvert w_{\varepsilon,k}\rvert)\,dy
\\
&=
C
\int_{\varepsilon(k+Y^{*})}
\Phi(C\lvert v\rvert)\,dx.
\end{align*}
The complete cells are disjoint.  Summation over the explicit index set
\(K_{\varepsilon}\), followed by addition of the identity part on
\(R_{\varepsilon}\), proves \eqref{eq:uniform-extension-zero}.

For the gradients, the linearity of \(P\) allows us to apply
\eqref{eq:cell-extension-one} to \(w_{\varepsilon,k}/\varepsilon\). Indeed, for \(i=1,\ldots,N\),
\[
D_{y_i}\left(\frac{w_{\varepsilon,k}(y)}\varepsilon\right)
 =D_{x_i}v(\varepsilon(k+y)),\qquad
D_{y_i}P\!\left(\frac{w_{\varepsilon,k}}\varepsilon\right)
 =\frac1\varepsilon D_{y_i}Pw_{\varepsilon,k}.
\]
Consequently, for each \(k\in K_\varepsilon\),
\begin{align*}
&\int_{\varepsilon(k+Y)}\sum_{i=1}^N
 \Phi(|D_{x_i}P_\varepsilon v|)\,dx\\
&\quad=\varepsilon^N\int_Y\sum_{i=1}^N
 \Phi\!\left(\left|D_{y_i}P
  \left(\frac{w_{\varepsilon,k}}\varepsilon\right)\right|\right)dy\\
&\quad\le C\varepsilon^N\int_{Y^*}\sum_{i=1}^N
 \Phi\!\left(C\left|D_{y_i}
 \left(\frac{w_{\varepsilon,k}}\varepsilon\right)\right|\right)dy\\
&\quad=C\int_{\varepsilon(k+Y^*)}\sum_{i=1}^N
 \Phi(C|D_{x_i}v|)\,dx.
\end{align*}
Summation over the disjoint complete cells, together with the identity
\(P_\varepsilon v=v\) on \(R_\varepsilon\), proves
\eqref{eq:uniform-extension-one}, with constants independent of
\(\varepsilon\).
\end{proof}

\begin{proposition}
\label{prop:scaled-trace}
There is \(C>0\), independent of \(\varepsilon\), such that every
\(v\in W^{1}L^{\Phi}(\Omega_{\varepsilon})\) satisfies
\begin{equation}
\varepsilon
\int_{\Gamma_{\varepsilon}}
\Phi(\lvert v\rvert)\,d\sigma_{x}
\leq
C
\int_{\Omega_{\varepsilon}}
\left[
\Phi(C\lvert v\rvert)
+
\varepsilon^{p^{-}}
\sum_{i=1}^{N}
\Phi(C\lvert D_{x_{i}}v\rvert)
\right]\,dx.
\label{eq:scaled-trace-index}
\end{equation}
\end{proposition}

\begin{proof}
For \(k\in K_{\varepsilon}\), apply \eqref{eq:cell-trace} to
\(
w_{\varepsilon,k}(y)
=
v\bigl(\varepsilon(k+y)\bigr).
\)
After the changes of variables on the surface and in the volume,
\begin{align*}
\varepsilon
\int_{\varepsilon(k+\partial T)}
\Phi(\lvert v\rvert)\,d\sigma_{x}
 &
=
\varepsilon^{N}
\int_{\partial T}
\Phi(\lvert\gamma_{0}w_{\varepsilon,k}\rvert)\,d\sigma_{y}
\\
&
\leq
C\varepsilon^{N}
\int_{Y^{*}}
\left[
\Phi(C\lvert w_{\varepsilon,k}\rvert)
+
\sum_{i=1}^{N}
\Phi(C\lvert D_{y_{i}}w_{\varepsilon,k}\rvert)
\right]\,dy
\\
&
=
C
\int_{\varepsilon(k+Y^{*})}
\left[
\Phi(C\lvert v\rvert)
+
\sum_{i=1}^{N}
\Phi(C\varepsilon\lvert D_{x_{i}}v\rvert)
\right]\,dx.
\end{align*}
Summing over \(k\in K_{\varepsilon}\) gives
\[
\varepsilon
\int_{\Gamma_{\varepsilon}}
\Phi(\lvert v\rvert)\,d\sigma_{x}
\leq
C
\int_{\Omega_{\varepsilon}}
\left[
\Phi(C\lvert v\rvert)
+
\sum_{i=1}^{N}
\Phi(C\varepsilon\lvert D_{x_{i}}v\rvert)
\right]\,dx.
\]
Since \(0<\varepsilon<1\), \eqref{eq:scaling-small} gives
\[
\Phi(C\varepsilon\lvert D_{x_{i}}v\rvert)
\leq
\varepsilon^{p^{-}}
\Phi(C\lvert D_{x_{i}}v\rvert),
\qquad i=1,\ldots,N,
\]
and hence \eqref{eq:scaled-trace-index}.
\end{proof}

\subsection{Nonlinear surface compactness}
\label{subsec:surface-compactness}

The trace bounds above control surface modulars by volume quantities, but they do not by themselves yield convergence of nonlinear surface terms.  This is obtained by comparing traces with averages over complete cells.  For \(U_{\varepsilon}\in W^{1}L^{\Phi}(\Omega)\), define the complete-cell averages
\begin{equation}
m_{\varepsilon,k}
=
\frac{1}{\varepsilon^{N}\lvert Y\rvert}
\int_{\varepsilon(k+Y)}
U_{\varepsilon}(x)\,dx,
\qquad k\in K_{\varepsilon},
\label{eq:cell-mean}
\end{equation}
and
\begin{equation}
Q_{\varepsilon}U_{\varepsilon}(x)
=
\left\{
\begin{aligned}
m_{\varepsilon,k},
&\quad
x\in\varepsilon(k+Y),\ k\in K_{\varepsilon},
\\
U_{\varepsilon}(x),
&\quad
x\in R_{\varepsilon}.
\end{aligned}
\right.
\label{eq:cell-average-operator}
\end{equation}

\begin{proposition}
\label{prop:nonlinear-surface-compactness}
Let \(v_{\varepsilon}\in W^{1}L^{\Phi}(\Omega_{\varepsilon})\), \(\varepsilon\in E\), and let \(U_{\varepsilon}\in W^{1}L^{\Phi}(\Omega)\) satisfy \(U_{\varepsilon}=v_{\varepsilon}\) almost everywhere in \(\Omega_{\varepsilon}\),
\begin{equation}
\sup_{\varepsilon\in E}
\int_{\Omega}
\left[
\Phi(\lvert U_{\varepsilon}\rvert)
+
\sum_{i=1}^{N}
\Phi(\lvert D_{x_{i}}U_{\varepsilon}\rvert)
\right]\,dx
<\infty,
\qquad
U_{\varepsilon}\rightharpoonup v_{0}
\quad\text{in }W^{1}L^{\Phi}(\Omega).
\label{eq:surface-compactness-assumption}
\end{equation}
Then
\begin{align}
\varepsilon
\int_{\Gamma_{\varepsilon}}
\rho\left(\frac{x}{\varepsilon}\right)
\Phi(\lvert v_{\varepsilon}\rvert)\,d\sigma_{x}
&\longrightarrow
\beta_{\rho}
\int_{\Omega}\Phi(\lvert v_{0}\rvert)\,dx,
\label{eq:surface-Phi-limit}
\\
\varepsilon
\int_{\Gamma_{\varepsilon}}
\rho\left(\frac{x}{\varepsilon}\right)
\phi(\lvert v_{\varepsilon}\rvert)
\lvert v_{\varepsilon}\rvert\,d\sigma_{x}
&\longrightarrow
\beta_{\rho}
\int_{\Omega}
\phi(\lvert v_{0}\rvert)\lvert v_{0}\rvert\,dx,
\label{eq:surface-H-limit}
\end{align}
and, for every \(\psi\in C(\overline{\Omega})\),
\begin{equation}
\varepsilon
\int_{\Gamma_{\varepsilon}}
\rho\left(\frac{x}{\varepsilon}\right)
\phi(\lvert v_{\varepsilon}\rvert)
\operatorname{sgn}(v_{\varepsilon})
\psi(x)\,d\sigma_{x}
\longrightarrow
\beta_{\rho}
\int_{\Omega}
\phi(\lvert v_{0}\rvert)
\operatorname{sgn}(v_{0})
\psi(x)\,dx.
\label{eq:surface-derivative-limit}
\end{equation}
In particular, these conclusions hold for \(v_{\varepsilon}\in X_{\varepsilon}\) bounded in \(W^{1}L^{\Phi}(\Omega_{\varepsilon})\), with \(U_{\varepsilon}=P_{\varepsilon}v_{\varepsilon}\rightharpoonup v_{0}\) in \(W_{0}^{1}L^{\Phi}(\Omega)\).
\end{proposition}

\begin{proof}
The compact embedding \eqref{eq:compact-volume-embedding} gives
\begin{equation}
U_{\varepsilon}
\longrightarrow
v_{0}
\quad\text{in }L^{\Phi}(\Omega).
\label{eq:extension-strong-LPhi}
\end{equation}

We first estimate the difference between each trace and its complete-cell average.
Apply \eqref{eq:cell-trace-Poincare} to
\[
W_{\varepsilon,k}(y)
=
U_{\varepsilon}\bigl(\varepsilon(k+y)\bigr).
\]
Its average over \(Y\) is \(m_{\varepsilon,k}\).  Since
\[
\gamma_{0}U_{\varepsilon}
=
\gamma_{0}v_{\varepsilon}
\quad\text{on }\varepsilon(k+\partial T),
\]
rescaling by \eqref{eq:one-cell-surface-Riemann} gives
\begin{align*}
&\varepsilon
\int_{\varepsilon(k+\partial T)}
\Phi
\left(
\lvert v_{\varepsilon}-m_{\varepsilon,k}\rvert
\right)\,d\sigma_{x}
\\
&\quad
\leq
C
\int_{\varepsilon(k+Y)}
\sum_{i=1}^{N}
\Phi(C\varepsilon\lvert D_{x_{i}}U_{\varepsilon}\rvert)\,dx.
\end{align*}
Sum for \(k\in K_{\varepsilon}\) and use
\eqref{eq:scaling-small}.  We obtain
\begin{equation}
\int_{\Gamma_{\varepsilon}}
\Phi
\left(
\left|
v_{\varepsilon}
-Q_{\varepsilon}U_{\varepsilon}
\right|
\right)\,d\vartheta_{\varepsilon}
\leq
C\varepsilon^{p^{-}}
\int_{\Omega}
\sum_{i=1}^{N}
\Phi(C\lvert D_{x_{i}}U_{\varepsilon}\rvert)\,dx
\longrightarrow0.
\label{eq:trace-close-to-mean}
\end{equation}

The cell averages converge to the same strong volume limit.
Lemma~\ref{lem:cell-average}, after rescaling, gives
\begin{equation}
\int_{\Omega_{\varepsilon}^{\mathrm{cell}}}
\Phi
\left(
\left|
U_{\varepsilon}
-Q_{\varepsilon}U_{\varepsilon}
\right|
\right)\,dx
\leq
C\varepsilon^{p^{-}}
\int_{\Omega}
\sum_{i=1}^{N}
\Phi(C\lvert D_{x_{i}}U_{\varepsilon}\rvert)\,dx
\longrightarrow0.
\label{eq:volume-close-to-mean}
\end{equation}
The two functions coincide on \(R_{\varepsilon}\).  Combining
\eqref{eq:extension-strong-LPhi},
\eqref{eq:volume-close-to-mean} and the modular--norm equivalence yields
\begin{equation}
Q_{\varepsilon}U_{\varepsilon}
\longrightarrow
v_{0}
\quad\text{in }L^{\Phi}(\Omega).
\label{eq:mean-strong-LPhi}
\end{equation}

We next pass to the nonlinear surface modulars.
Let either
\[
\mathcal{H}(t)=\Phi(t)
\quad\text{or}\quad
\mathcal{H}(t)=t\phi(t).
\]
Because \(Q_{\varepsilon}U_{\varepsilon}\) is constant on every complete cell,
\begin{align}
\varepsilon
\int_{\Gamma_{\varepsilon}}
\rho\left(\frac{x}{\varepsilon}\right)
\mathcal{H}(\lvert Q_{\varepsilon}U_{\varepsilon}\rvert)\,d\sigma_{x}
\nonumber
&
=
\sum_{k\in K_{\varepsilon}}
\varepsilon^{N}
\mathcal{H}(\lvert m_{\varepsilon,k}\rvert)
\int_{\partial T}\rho(y)\,d\sigma_{y}
\nonumber\\
&
=
\beta_{\rho}
\int_{\Omega_{\varepsilon}^{\mathrm{cell}}}
\mathcal{H}(\lvert Q_{\varepsilon}U_{\varepsilon}\rvert)\,dx.
\label{eq:exact-surface-volume-identity}
\end{align}
By \eqref{eq:Phi-indices},
\[
0\leq t\phi(t)\leq p^{+}\Phi(t).
\]
By \eqref{eq:mean-strong-LPhi} and
Lemma~\ref{lem:nonlinear-modular-continuity},
\(\mathcal{H}(|Q_\varepsilon U_\varepsilon|)\to \mathcal{H}(|v_0|)\) in
\(L^1(\Omega)\) for either choice of \(H\). Therefore,
\[
\int_{R_\varepsilon}\mathcal{H}(|Q_\varepsilon U_\varepsilon|)\,dx
\le\|\mathcal{H}(|Q_\varepsilon U_\varepsilon|)-\mathcal{H}(|v_0|)\|_{L^1(\Omega)}
+\int_{R_\varepsilon}\mathcal{H}(|v_0|)\,dx\longrightarrow0,
\]
where the last term tends to zero by \eqref{eq:boundary-layer-small}
and absolute continuity of the integral. Hence the last member of
\eqref{eq:exact-surface-volume-identity} converges to
\[
\beta_{\rho}\int_{\Omega}\mathcal{H}(\lvert v_{0}\rvert)\,dx.
\]

To apply Lemma~\ref{lem:nonlinear-modular-continuity} on the varying
measure spaces \((\Gamma_\varepsilon,\vartheta_\varepsilon)\), let
\(a_\varepsilon=v_\varepsilon\) and
\(b_\varepsilon=Q_\varepsilon U_\varepsilon\). The measure bound is
\eqref{eq:one-cell-surface-Riemann}, and the modular difference tends
to zero by \eqref{eq:trace-close-to-mean}. For every fixed \(L>0\),
the identity \eqref{eq:exact-surface-volume-identity} with \(\rho\equiv1\), in which \(\beta_{\rho}\) becomes \(\lvert\partial T\rvert\), and Jensen's inequality give
\begin{align*}
\int_{\Gamma_\varepsilon}
 \Phi(L|b_\varepsilon|)\,d\vartheta_\varepsilon
&=|\partial T|\sum_{k\in K_\varepsilon}
 \varepsilon^N\Phi(L|m_{\varepsilon,k}|)\\
&\le|\partial T|\int_\Omega\Phi(L|U_\varepsilon|)\,dx
\le C_L.
\end{align*}
Similarly, the scaled trace inequality \eqref{eq:scaled-trace-index}
and \eqref{eq:fixed-scaling} yield
\[
\int_{\Gamma_\varepsilon}
\Phi(L|a_\varepsilon|)\,d\vartheta_\varepsilon
\le C_L\int_{\Omega_\varepsilon}
\left[\Phi(C_L|v_\varepsilon|)
 +\varepsilon^{p^-}\sum_{i=1}^N
   \Phi(C_L|D_{x_i}v_\varepsilon|)\right]dx
\le C'_L.
\]
Both bounds are uniform in \(\varepsilon\), by
\eqref{eq:surface-compactness-assumption} and \eqref{eq:fixed-scaling};
thus all the hypotheses of part~(a) of the Nemytskii lemma hold.
It follows that
\[
\int_{\Gamma_{\varepsilon}}
\left|
\mathcal{H}(\lvert v_{\varepsilon}\rvert)
-\mathcal{H}(\lvert Q_{\varepsilon}U_{\varepsilon}\rvert)
\right|\,d\vartheta_{\varepsilon}
\longrightarrow0.
\]
Since \(0\leq\rho\leq\rho_{+}\), the same convergences hold with the weight \(\rho(x/\varepsilon)\).
This proves \eqref{eq:surface-Phi-limit} and
\eqref{eq:surface-H-limit}.

The same approximation also identifies the limit of the derivative of the surface modular.
Lemma~\ref{lem:nonlinear-modular-continuity} also gives
\begin{equation}
\int_{\Gamma_{\varepsilon}}
\left|
\phi(\lvert v_{\varepsilon}\rvert)
\operatorname{sgn}(v_{\varepsilon})
-
\phi(\lvert Q_{\varepsilon}U_{\varepsilon}\rvert)
\operatorname{sgn}(Q_{\varepsilon}U_{\varepsilon})
\right|\,d\vartheta_{\varepsilon}
\longrightarrow0.
\label{eq:Euler-trace-replacement}
\end{equation}
Let \(g(t)=\phi(|t|)\operatorname{sgn}(t)\) be as in
Lemma~\ref{lem:nonlinear-modular-continuity}.
For \(x\in\varepsilon(k+Y)\), \(k\in K_\varepsilon\), define the
piecewise constant function
\(\psi_\varepsilon^0(x)=\psi(\varepsilon k)\). Uniform continuity on
\(\overline\Omega\) gives
\begin{equation}
\|\psi_\varepsilon^0-\psi\|_{L^\infty(
\Omega_\varepsilon^{\mathrm{cell}})}
\leq \omega_\psi(\sqrt N\,\varepsilon)
\longrightarrow0,
\label{eq:piecewise-test-uniform}
\end{equation}
where \(\omega_\psi\) is the modulus of continuity of \(\psi\).
Moreover, \eqref{eq:mean-strong-LPhi} and the last conclusion of
Lemma~\ref{lem:nonlinear-modular-continuity} imply
\begin{equation}
g(Q_\varepsilon U_\varepsilon)
\longrightarrow g(v_0)
\quad\text{in }L^1(\Omega).
\label{eq:Euler-volume-strong}
\end{equation}
Consequently, using \eqref{eq:piecewise-test-uniform},
\eqref{eq:boundary-layer-small}, and absolute continuity of the integral,
\begin{align*}
\int_{\Omega_\varepsilon^{\mathrm{cell}}}
g(Q_\varepsilon U_\varepsilon)\psi_\varepsilon^0\,dx
&\longrightarrow
\int_\Omega g(v_0)\psi\,dx.
\end{align*}
Indeed, the difference is bounded by the \(L^1\)-error in
\eqref{eq:Euler-volume-strong}, the product of
\(\omega_\psi(\sqrt N\varepsilon)\) with a uniformly bounded
\(L^1\)-norm, and \(\int_{R_\varepsilon}|g(v_0)\psi|\,dx\).

The discrete expression is exactly this volume integral:
\begin{align}
&\varepsilon
\int_{\Gamma_\varepsilon}
\rho\left(\frac{x}{\varepsilon}\right)
g(Q_\varepsilon U_\varepsilon)
\psi_\varepsilon^0(x)\,d\sigma_x
\nonumber\\
&\qquad=
\beta_\rho\sum_{k\in K_\varepsilon}\varepsilon^N
g(m_{\varepsilon,k})\psi(\varepsilon k)
=
\beta_\rho
\int_{\Omega_\varepsilon^{\mathrm{cell}}}
g(Q_\varepsilon U_\varepsilon)\psi_\varepsilon^0\,dx.
\label{eq:Euler-discrete-volume}
\end{align}
Replacing \(\psi_\varepsilon^0\) by \(\psi\) on the surface produces an
error bounded by
\[
\beta_\rho\,\omega_\psi(\sqrt N\varepsilon)
\int_{\Omega_\varepsilon^{\mathrm{cell}}}
|g(Q_\varepsilon U_\varepsilon)|\,dx,
\]
which tends to zero by \eqref{eq:Euler-volume-strong}. Finally,
\eqref{eq:Euler-trace-replacement}, multiplied by the bounded factor
\(\rho(x/\varepsilon)\psi(x)\), replaces
\(g(Q_\varepsilon U_\varepsilon)\) by \(g(v_\varepsilon)\) on the
surface. This proves \eqref{eq:surface-derivative-limit} and completes
the proof.
\end{proof}

\subsection{Compactness in perforated domains}
\label{subsec:perforated-compactness}

We restrict weak two-scale limits to the solid phase, and combine this with the uniform extension to describe the limits of gradients in \(\Omega_{\varepsilon}\).

\begin{lemma}
\label{lem:solid-phase-restriction}
Let \(z_\varepsilon\) be bounded in \(L^\Phi(\Omega)\) and weakly two-scale convergent to \(z\in L^\Phi(\Omega\times Y)\). Then
\[
\chi_\varepsilon z_\varepsilon\overset{2s}{\rightharpoonup}\chi_{Y^*}z.
\]
The assertion also holds with \(\Phi\) replaced by \(\widetilde\Phi\).
\end{lemma}
\begin{proof}
It suffices to test against
\(\psi\in C_c(\Omega;C_{\mathrm{per}}(Y))\), and then use density and
the uniform Orlicz bounds. Choose smooth periodic functions
\(\eta_\delta\), \(0\leq\eta_\delta\leq1\), such that
\[
\|\eta_\delta-\chi_{Y^*}\|_{\widetilde\Phi,Y}
\longrightarrow0
\quad\text{as }\delta\downarrow0.
\]
On \(\Omega_\varepsilon^{\mathrm{cell}}\) one has
\(\chi_\varepsilon=\chi_{Y^*}(x/\varepsilon)\); hence the triangle
inequality gives the explicit boundary-layer estimate
\begin{align}
\|\chi_\varepsilon-\eta_\delta(x/\varepsilon)\|_{\widetilde\Phi,\Omega}
&\leq
\|[\chi_{Y^*}(x/\varepsilon)-\eta_\delta(x/\varepsilon)]
\mathbf1_{\Omega_\varepsilon^{\mathrm{cell}}}\|_{\widetilde\Phi,\Omega}
 +2\|\mathbf1_{R_\varepsilon}\|_{\widetilde\Phi,\Omega}.
\label{eq:solid-restriction-boundary-layer}
\end{align}
The last term tends to zero because \(|R_\varepsilon|\to0\) and the
Luxemburg norm is absolutely continuous under \(\Delta_2\). For the first
term, the periodic mean-value property
\cite[Chapter~2]{CioranescuDonato1999}, applied to
\(\widetilde\Phi(|\chi_{Y^*}-\eta_\delta|/a)\) for each \(a>0\), and
the \(\Delta_2\)-condition imply
\begin{equation}
\limsup_{\varepsilon\to0}
\|\chi_\varepsilon-\eta_\delta(x/\varepsilon)\|_{\widetilde\Phi,\Omega}
\leq
C\|\chi_{Y^*}-\eta_\delta\|_{\widetilde\Phi,Y}.
\label{eq:solid-restriction-norm}
\end{equation}
H\"older's inequality \eqref{eq:Orlicz-Holder} therefore bounds the error in the pairing with
\(z_\varepsilon\psi(x,x/\varepsilon)\) by a constant times the right-hand
side of \eqref{eq:solid-restriction-norm}. For fixed \(\delta\),
\(\eta_\delta(y)\psi(x,y)\) is an element of
\(L^{\widetilde{\Phi}}(\Omega;C_{\mathrm{per}}(Y))\). We may thus
pass first to \(\varepsilon\downarrow0\), and then to \(\delta\downarrow0\).
The same H\"older estimate on \(\Omega\times Y\) replaces
\(\eta_\delta\) by \(\chi_{Y^*}\) in the limit and proves the assertion.
For \(z_\varepsilon\) bounded in \(L^{\widetilde\Phi}(\Omega)\), the same
argument applies after interchanging the complementary pair
\(\Phi,\widetilde\Phi\) and choosing the smooth approximation in
\(L^\Phi(Y)\).
\end{proof}

\begin{proposition}
\label{prop:perforated-gradient-compactness}
Let \(v_{\varepsilon}\in X_{\varepsilon}\), \(\varepsilon\in E\), be bounded in \(W^{1}L^{\Phi}(\Omega_{\varepsilon})\).  There exist a subsequence, not relabelled, a function
\(v_{0}\in W_{0}^{1}L^{\Phi}(\Omega)\) and a function
\(v_{1}\in L^{1}(\Omega;W_{\#}^{1}L^{\Phi}(Y^{*}))\) with
\(\nabla_{y}v_{1}\in L^{\Phi}(\Omega\times Y^{*})^{N}\), such that
\begin{align}
P_{\varepsilon}v_{\varepsilon}
&\rightharpoonup
v_{0}
&&\text{in }W_{0}^{1}L^{\Phi}(\Omega),
\label{eq:gradient-compactness-weak}
\\
P_{\varepsilon}v_{\varepsilon}
&\longrightarrow
v_{0}
&&\text{in }L^{\Phi}(\Omega),
\label{eq:gradient-compactness-strong}
\\
\chi_{\varepsilon}D_{x_{i}}v_{\varepsilon}
&\overset{2s}{\rightharpoonup}
\chi_{Y^{*}}
\left(
D_{x_{i}}v_{0}
+D_{y_{i}}v_{1}
\right),
&&i=1,\ldots,N.
\label{eq:perforated-gradient-limit}
\end{align}
\end{proposition}

\begin{proof}
By Proposition~\ref{prop:uniform-extension},
\(U_{\varepsilon}=P_{\varepsilon}v_{\varepsilon}\) is bounded in
\(W_{0}^{1}L^{\Phi}(\Omega)\). Reflexivity and
\eqref{eq:compact-volume-embedding} give
\eqref{eq:gradient-compactness-weak}--\eqref{eq:gradient-compactness-strong}
after extraction. The gradient two-scale compactness theorem of
\cite{TachagoNnang2012} provides, after a further extraction,
\(\widehat{v}_{1}\in L^{1}(\Omega;W_{\#}^{1}L^{\Phi}(Y))\), with
\(\nabla_{y}\widehat{v}_{1}\in L^{\Phi}(\Omega\times Y)^{N}\), such that
\[
D_{x_i}U_\varepsilon
\overset{2s}{\rightharpoonup}
D_{x_i}v_0+D_{y_i}\widehat v_1,
\qquad i=1,\ldots,N.
\]
Define the restriction and renormalization operator
\[
\mathcal R:W_\#^1L^\Phi(Y)\longrightarrow W_\#^1L^\Phi(Y^*),
\qquad
\mathcal Rw
=w|_{Y^*}-\frac1{|Y^*|}\int_{Y^*}w(y)\,dy.
\]
Because \(\overline T\Subset Y\), restriction preserves the periodic
traces on opposite faces of \(Y\). Moreover,
\[
\sum_{i=1}^N\|D_{y_i}\mathcal Rw\|_{\Phi,Y^*}
\leq
\sum_{i=1}^N\|D_{y_i}w\|_{\Phi,Y},
\]
so \(\mathcal R\) is bounded and linear for the gradient norms used on
these mean-zero spaces. Hence
\(v_1(x,\cdot)=\mathcal R\widehat v_1(x,\cdot)\) is strongly measurable,
belongs to \(L^1(\Omega;W_\#^1L^\Phi(Y^*))\), and satisfies
\(\nabla_yv_1\in L^\Phi(\Omega\times Y^*)^N\). Finally,
\(\chi_\varepsilon D_{x_i}v_\varepsilon
=\chi_\varepsilon D_{x_i}U_\varepsilon\), and
Lemma~\ref{lem:solid-phase-restriction} gives
\eqref{eq:perforated-gradient-limit}.
\end{proof}

\section{Homogenization procedure}
\label{sec:homogenization}

\subsection{Existence and uniform estimates}
\label{subsec:estimates}

Two elementary facts serve both the microscopic and the homogenized problem: normalization by dilation, and the constrained minimization principle of Lemma~\ref{lem:constrained-minimization}.  For the first, let \(\mathcal{M}\) denote \(\mathcal{M}_{\varepsilon}\) or \(\mathcal{M}_{0}\), and let \(v\) satisfy \(\mathcal{M}(v)>0\).  By \eqref{eq:scaling-large}--\eqref{eq:scaling-small}, the map \(t\mapsto\mathcal{M}(tv)\) is continuous and strictly increasing from \(0\) to \(+\infty\), and the unique \(t>0\) with \(\mathcal{M}(tv)=\alpha\) satisfies
\begin{equation}
\min\bigl\{r^{1/p^{-}},r^{1/p^{+}}\bigr\}
\leq t\leq
\max\bigl\{r^{1/p^{-}},r^{1/p^{+}}\bigr\},
\qquad r=\frac{\alpha}{\mathcal{M}(v)}.
\label{eq:normalization-scaling}
\end{equation}
For the microscopic constraint, choose \(k\in K_\varepsilon\) and
\(z_\varepsilon\in C_c^\infty(\Omega)\) such that
\(z_\varepsilon=1\) on a neighbourhood of
\(\varepsilon(k+\partial T)\). Then
\[
\mathcal M_\varepsilon(z_\varepsilon)
\geq
\varepsilon^N\Phi(1)\int_{\partial T}\rho(y)\,d\sigma_y
>0.
\]
For the limit constraint, every nonzero nonnegative
\(z\in C_c^\infty(\Omega)\) satisfies \(\mathcal M_0(z)>0\).
Applying the preceding dilation to these functions shows explicitly that
the constraint sets in \eqref{eq:intro-minimum} and
\eqref{eq:intro-mu0} are nonempty.

\begin{lemma}
\label{lem:constrained-minimization}
Let \(X\) be a reflexive Banach space and let \(\mathcal{E},\mathcal{M}\in C^{1}(X;[0,\infty))\), with \(\mathcal{E}\) convex, \(\mathcal{M}\) sequentially weakly continuous, \(\{\mathcal{E}\leq c\}\) bounded for every \(c\), and
\(\mathcal{M}'(v)[v]\geq p^{-}\mathcal{M}(v)\) for every \(v\in X\).  If \(\alpha>0\) and \(\{\mathcal{M}=\alpha\}\neq\varnothing\), then \(\inf\{\mathcal{E}(v):\mathcal{M}(v)=\alpha\}\) is attained, and for every minimizer \(u\) there is a unique \(\lambda\in\mathbb{R}\) such that, for every \(v\in X\),
\begin{equation}
\mathcal{E}'(u)[v]=\lambda\,\mathcal{M}'(u)[v],
\qquad
\lambda=\frac{\mathcal{E}'(u)[u]}{\mathcal{M}'(u)[u]}.
\label{eq:abstract-multiplier}
\end{equation}
\end{lemma}

\begin{proof}
A minimizing sequence is bounded; by reflexivity, a subsequence converges weakly to some \(u\), which satisfies \(\mathcal{M}(u)=\alpha\) by weak continuity and is a minimizer because \(\mathcal{E}\) is convex and continuous, hence weakly lower semicontinuous.  Since \(\mathcal{M}'(u)[u]\geq p^{-}\alpha>0\), the level set is regular at \(u\), and the Lagrange multiplier rule for \(C^{1}\) constraints in Banach spaces \cite{Zeidler1985} gives \eqref{eq:abstract-multiplier}; testing with \(v=u\) identifies \(\lambda\), which is therefore unique.
\end{proof}

\begin{remark}
\label{rem:multipliers}
By Lemma~\ref{lem:constrained-minimization}, each minimizer \(u\) of \eqref{eq:intro-minimum} determines a unique multiplier \(\lambda_{\varepsilon}[u]=\mathcal{E}_{\varepsilon}'(u)[u]/\mathcal{M}_{\varepsilon}'(u)[u]\), given by \eqref{eq:intro-multiplier-ratio}; similarly, each minimizer \(u\) of \eqref{eq:intro-mu0} determines a unique multiplier \(\lambda_{0}[u]\) (Proposition~\ref{prop:homogenized-existence}).  When the minimizers are not unique, these multipliers may depend on \(u\).  Throughout, \(\lambda_{1,\varepsilon}(\alpha)=\lambda_{\varepsilon}[u_{\varepsilon,\alpha}]\) and \(\lambda_{1,0}(\alpha)=\lambda_{0}[u_{0,\alpha}]\) denote the multipliers of the selected minimizers.
\end{remark}

\begin{theorem}
\label{thm:micro-existence}
For every \(\alpha>0\) and \(\varepsilon\in E\), the minimum \(\mu_{1,\varepsilon}(\alpha)\) is attained, every minimizer gives a normalized principal eigencouple, and a minimizer may be chosen nonnegative.
\end{theorem}

\begin{proof}
Apply Lemma~\ref{lem:constrained-minimization} with \(X=X_{\varepsilon}\).  Since \(P_{\varepsilon}v=v\) in \(\Omega_{\varepsilon}\), the Poincar\'e inequality \eqref{eq:Orlicz-Poincare} in \(\Omega\) and Proposition~\ref{prop:uniform-extension} give
\[
\int_{\Omega_{\varepsilon}}\Phi(\lvert v\rvert)\,dx
\leq
\int_{\Omega}\Phi\bigl(d_{\Omega}\lvert\nabla P_{\varepsilon}v\rvert\bigr)\,dx
\leq
C\int_{\Omega_{\varepsilon}}\sum_{i=1}^{N}\Phi(C\lvert D_{x_{i}}v\rvert)\,dx,
\]
with \(C\) independent of \(\varepsilon\); together with \eqref{eq:F-growth}, \eqref{eq:fixed-scaling} and \eqref{eq:modular-small-norm}, this bounds the sublevel sets of \(\mathcal{E}_{\varepsilon}\).  By \eqref{eq:A-growth}, \eqref{eq:phi-complement-bound}, Lemma~\ref{lem:nonlinear-modular-continuity} and the continuity of the trace, \(\mathcal{E}_{\varepsilon}\) and \(\mathcal{M}_{\varepsilon}\) are of class \(C^{1}\), with derivatives given by the two sides of \eqref{eq:intro-weak}.  Proposition~\ref{prop:compact-Orlicz-trace} and Lemma~\ref{lem:nonlinear-modular-continuity} show that \(\mathcal{M}_{\varepsilon}\) is sequentially weakly continuous, and \eqref{eq:Phi-indices} gives \(\mathcal{M}_{\varepsilon}'(v)[v]\geq p^{-}\mathcal{M}_{\varepsilon}(v)\).  Finally, \(\lvert u\rvert\in X_{\varepsilon}\), \(\nabla\lvert u\rvert=\operatorname{sgn}(u)\nabla u\) almost everywhere, and \eqref{eq:F-even} show that \(\lvert u\rvert\) is a minimizer whenever \(u\) is.
\end{proof}

\begin{theorem}
\label{thm:uniform-estimates}
Fix \(\alpha>0\), and for every \(\varepsilon\in E\) let \(u_{\varepsilon,\alpha}\) be a nonnegative normalized minimizer.  There are \(\varepsilon_{0}>0\) and positive constants \(c_{\alpha},C_{\alpha}\) such that, for every \(\varepsilon\in E\) with \(\varepsilon<\varepsilon_{0}\),
\begin{align}
c_{\alpha}
\leq
\mu_{1,\varepsilon}(\alpha)
&\leq
C_{\alpha},
\label{eq:uniform-mu}
\\
\int_{\Omega_{\varepsilon}}
\Phi(\lvert\nabla u_{\varepsilon,\alpha}\rvert)\,dx
&\leq
C_{\alpha},
\label{eq:uniform-gradient}
\\
\int_{\Omega}
\left[
\Phi(\lvert P_{\varepsilon}u_{\varepsilon,\alpha}\rvert)
+
\sum_{i=1}^{N}
\Phi(\lvert D_{x_{i}}P_{\varepsilon}u_{\varepsilon,\alpha}\rvert)
\right]\,dx
&\leq
C_{\alpha},
\label{eq:uniform-extension}
\\
c_{\alpha}
\leq
\lambda_{1,\varepsilon}(\alpha)
&\leq
C_{\alpha},
\label{eq:uniform-lambda}
\\
\varepsilon
\int_{\Gamma_{\varepsilon}}
\widetilde{\Phi}
\left(
\frac{
\phi(\lvert u_{\varepsilon,\alpha}\rvert)
}{C_{\alpha}}
\right)
d\sigma_{x}
&\leq
C_{\alpha}.
\label{eq:uniform-boundary-dual}
\end{align}
\end{theorem}

\begin{proof}
Fix \(0\leq z\in C_{c}^{\infty}(\Omega)\), \(z\not\equiv0\).  Proposition~\ref{prop:nonlinear-surface-compactness}, applied with \(v_{\varepsilon}=z|_{\Omega_{\varepsilon}}\) and \(U_{\varepsilon}=z\), gives \(\mathcal{M}_{\varepsilon}(z)\to\beta_{\rho}\int_{\Omega}\Phi(z)\,dx>0\).  By \eqref{eq:normalization-scaling}, the dilation factors \(t_{\varepsilon}\) with \(\mathcal{M}_{\varepsilon}(t_{\varepsilon}z)=\alpha\) are bounded, and
\(\mu_{1,\varepsilon}(\alpha)\leq\mathcal{E}_{\varepsilon}(t_{\varepsilon}z)\leq c_{1}\int_{\Omega}\Phi(t_{\varepsilon}\lvert\nabla z\rvert)\,dx\leq C_{\alpha}\).  The lower bound in \eqref{eq:F-growth} gives \eqref{eq:uniform-gradient}, and Proposition~\ref{prop:uniform-extension}, \eqref{eq:Orlicz-Poincare} and \eqref{eq:fixed-scaling} give \eqref{eq:uniform-extension}.

For the lower bound, the normalization, \eqref{eq:scaled-trace-index}, \eqref{eq:Orlicz-Poincare} applied to \(P_{\varepsilon}u_{\varepsilon,\alpha}\) and Proposition~\ref{prop:uniform-extension} give
\begin{align*}
\alpha
&\leq
\rho_{+}\varepsilon\int_{\Gamma_{\varepsilon}}\Phi(\lvert u_{\varepsilon,\alpha}\rvert)\,d\sigma_{x}
\leq
C\int_{\Omega}\Phi(C\lvert P_{\varepsilon}u_{\varepsilon,\alpha}\rvert)\,dx
+C\int_{\Omega_{\varepsilon}}\sum_{i=1}^{N}\Phi(C\lvert D_{x_{i}}u_{\varepsilon,\alpha}\rvert)\,dx
\\
&\leq
C\int_{\Omega_{\varepsilon}}\Phi(\lvert\nabla u_{\varepsilon,\alpha}\rvert)\,dx,
\end{align*}
where the last inequality also uses \eqref{eq:fixed-scaling} and \eqref{eq:vector-component-modular-equivalence}.  Hence \(\int_{\Omega_{\varepsilon}}\Phi(\lvert\nabla u_{\varepsilon,\alpha}\rvert)\,dx\geq c_{\alpha}>0\), and \eqref{eq:F-growth} gives the lower bound in \eqref{eq:uniform-mu}.

In \eqref{eq:intro-multiplier-ratio}, the denominator lies in \([p^{-}\alpha,p^{+}\alpha]\) by \eqref{eq:Phi-indices}, and the numerator lies between \(a_{0}\) and \(a_{1}\) times \(\int_{\Omega_{\varepsilon}}\Phi(\lvert\nabla u_{\varepsilon,\alpha}\rvert)\,dx\) by \eqref{eq:A-pairing}; this proves \eqref{eq:uniform-lambda}.  Finally, \eqref{eq:phi-complement-bound} gives \(\widetilde{\Phi}(\phi(\lvert u_{\varepsilon,\alpha}\rvert))\leq(p^{+}-1)\Phi(\lvert u_{\varepsilon,\alpha}\rvert)\) on \(\Gamma_{\varepsilon}\), and \eqref{eq:scaled-trace-index} with \eqref{eq:uniform-gradient}--\eqref{eq:uniform-extension} proves \eqref{eq:uniform-boundary-dual}.
\end{proof}

\subsection{Cell problem and homogenized problem}
\label{subsec:macroscopic}

We now construct the objects of the limit problem: the cell correctors, the effective cell law and the homogenized eigencouple.

\begin{lemma}
\label{lem:affine-periodic-coercivity}
There is \(C>0\) such that, for every
\(\xi\in\mathbb{R}^{N}\) and every
\(w\in W_{\#}^{1}L^{\Phi}(Y^{*})\),
\begin{equation}
\Phi(\lvert\xi\rvert)
+
\int_{Y^{*}}
\sum_{i=1}^{N}
\Phi(\lvert D_{y_{i}}w\rvert)\,dy
\leq
C
\int_{Y^{*}}
\sum_{i=1}^{N}
\Phi
\left(
C\lvert\xi_{i}+D_{y_{i}}w\rvert
\right)\,dy.
\label{eq:affine-periodic-coercivity}
\end{equation}
\end{lemma}

\begin{proof}
Consider the function
\[
\ell(y)=\xi\cdot y+w(y)
\quad\text{in }Y^{*}.
\]
Since \(\ell\in W^{1}L^{\Phi}(Y^{*})\), we may put
\(\widehat{\ell}=P\ell\in W^{1}L^{\Phi}(Y)\), where \(P\) is the extension operator of
Lemma~\ref{lem:cell-extension-trace}.  As \(\overline{T}\subset Y\) and \(P\ell=\ell\) in \(Y^{*}\), the function \(\widehat{\ell}\) coincides with \(\ell\) near \(\partial Y\); hence

\[
\widehat{\ell}(y+e_{i})-\widehat{\ell}(y)=\xi_{i}
\quad\text{on the corresponding opposite faces},
\qquad i=1,\ldots,N,
\]
and
\begin{equation}
\int_{Y}
\sum_{i=1}^{N}
\Phi(\lvert D_{y_{i}}\widehat{\ell}\rvert)\,dy
\leq
C
\int_{Y^{*}}
\sum_{i=1}^{N}
\Phi
\left(
C\lvert\xi_{i}+D_{y_{i}}w\rvert
\right)\,dy.
\label{eq:quasiperiodic-extension}
\end{equation}
by \eqref{eq:cell-extension-one} applied to \(\ell\), since \(D_{y_{i}}\ell=\xi_{i}+D_{y_{i}}w\) in \(Y^{*}\).  The function
\(\widehat{\ell}(y)-\xi\cdot y\) is periodic.  Hence, for every
\(i=1,\ldots,N\),
\[
\frac{1}{\lvert Y\rvert}
\int_{Y}D_{y_{i}}\widehat{\ell}(y)\,dy
=
\xi_{i}.
\]
Jensen's inequality, followed by convexity across the \(N\) components, gives
\[
\Phi(\lvert\xi\rvert)
\leq
C
\int_{Y}
\sum_{i=1}^{N}
\Phi(C\lvert D_{y_{i}}\widehat{\ell}\rvert)\,dy.
\]
Combine this with \eqref{eq:quasiperiodic-extension}.  Finally,
\[
\lvert D_{y_{i}}w\rvert
\leq
\lvert\xi_{i}+D_{y_{i}}w\rvert+\lvert\xi_{i}\rvert,
\qquad i=1,\ldots,N,
\]
and \(\Delta_{2}\) absorb the second term.  This proves
\eqref{eq:affine-periodic-coercivity}.
\end{proof}

\begin{proposition}
\label{prop:cell-minimizer}
For every \(\xi\in\mathbb{R}^{N}\), the infimum
\eqref{eq:intro-Fhom} is attained at a unique
\(w_{\xi}\in W_{\#}^{1}L^{\Phi}(Y^{*})\), characterized by the following property: for every \(\zeta\in W_{\#}^{1}L^{\Phi}(Y^{*})\),
\begin{equation}
\int_{Y^{*}}
\sum_{i=1}^{N}
A_{i}\bigl(y,\xi+\nabla_{y}w_{\xi}\bigr)
D_{y_{i}}\zeta\,dy
=0.
\label{eq:cell-variational-equation}
\end{equation}
The function \(F_{\mathrm{hom}}\) is even, convex and of class \(C^{1}\), and there are \(0<c\leq C\) such that, for every \(\xi\in\mathbb{R}^{N}\),
\begin{align}
c\Phi(\lvert\xi\rvert)
\leq
F_{\mathrm{hom}}(\xi)
&\leq
C\Phi(\lvert\xi\rvert),
\label{eq:Fhom-growth}
\\
D_{\xi}F_{\mathrm{hom}}(\xi)
&=
A_{\mathrm{hom}}(\xi),
\label{eq:Fhom-derivative}
\\
\widetilde\Phi\!\left(\frac{|A_{\mathrm{hom}}(\xi)|}{c_2}\right)
&\leq
\frac{c_1}{c_0}\,\Phi(|\xi|).
\label{eq:Ahom-complementary-growth}
\end{align}
Consequently, \(\mathcal E_0\in C^{1}(W_0^1L^\Phi(\Omega))\) and, for every \(u,v\in W_0^1L^\Phi(\Omega)\),
\begin{equation}
\mathcal E_0'(u)[v]
=\int_\Omega A_{\mathrm{hom}}(\nabla u)\cdot\nabla v\,dx.
\label{eq:homogenized-energy-C1}
\end{equation}
\end{proposition}

\begin{proof}
By \eqref{eq:F-growth} and Lemma~\ref{lem:affine-periodic-coercivity}, minimizing sequences are bounded in \(W_{\#}^{1}L^{\Phi}(Y^{*})\); reflexivity and weak lower semicontinuity of the convex integral give a minimizer.  If \(w,\widehat{w}\) are minimizers, strict convexity of \(F(y,\cdot)\) gives \(\nabla_{y}w=\nabla_{y}\widehat{w}\), and connectedness of \(Y^{*}\) with the zero means gives \(w=\widehat{w}\).  Put \(G_{\xi}=\xi+\nabla_{y}w_{\xi}\) and \(Z=\nabla_{y}\zeta\). For \(|t|\leq1\) and \(0\leq s\leq1\), Young's inequality, \eqref{eq:A-growth} and \(\Delta_2\) give
\[
|A(y,G_{\xi}+stZ)\cdot Z|
\leq \Phi(c_2|Z|)+\widetilde\Phi\left(\frac{|A(y,G_{\xi}+stZ)|}{c_2}\right)
\leq C\bigl[\Phi(|G_{\xi}|)+\Phi(|Z|)\bigr].
\]
The right-hand side belongs to \(L^1(Y^*)\). Hence the fundamental theorem of calculus and dominated convergence justify differentiating \(\int_{Y^*}F(y,G_{\xi}+tZ)\,dy\) at zero. Its minimality yields \eqref{eq:cell-variational-equation}.  The choice \(w=0\) gives the upper bound in \eqref{eq:Fhom-growth}, \eqref{eq:F-growth} and Lemma~\ref{lem:affine-periodic-coercivity} give the lower bound, the competitor \(\theta w_{\xi}+(1-\theta)w_{\eta}\) at \(\theta\xi+(1-\theta)\eta\) gives convexity, and the substitution \((\xi,w)\mapsto(-\xi,-w)\) gives evenness.

Conversely, let \(\overline w\in W_\#^1L^\Phi(Y^*)\) satisfy
\eqref{eq:cell-variational-equation} with \(\overline w\) in place of
\(w_\xi\). For every \(w\in W_\#^1L^\Phi(Y^*)\), convexity gives
\[
\int_{Y^*}F(y,\xi+\nabla_yw)\,dy
\geq
\int_{Y^*}F(y,\xi+\nabla_y\overline w)\,dy
+\int_{Y^*}A(y,\xi+\nabla_y\overline w)
 \cdot\nabla_y(w-\overline w)\,dy.
\]
The last integral vanishes, so \(\overline w\) is a minimizer and hence,
by the uniqueness already proved, \(\overline w=w_\xi\). Thus
\eqref{eq:cell-variational-equation} characterizes the unique cell
minimizer.

Fix \(\xi,h\in\mathbb{R}^{N}\) and put \(G_{\xi}=\xi+\nabla_{y}w_{\xi}\).  Since \(w_{\xi}\) is admissible at \(\xi+th\),
\[
F_{\mathrm{hom}}(\xi+th)-F_{\mathrm{hom}}(\xi)
\leq
\Upsilon(t)-\Upsilon(0),
\qquad
\Upsilon(t)=\int_{Y^{*}}F(y,G_{\xi}+th)\,dy.
\]
The convex function \(\Upsilon\) is differentiable, with \(\Upsilon'(0)=A_{\mathrm{hom}}(\xi)\cdot h\), by the preceding domination with \(Z=h\).  Hence the directional derivative satisfies \(F_{\mathrm{hom}}'(\xi;h)\leq A_{\mathrm{hom}}(\xi)\cdot h\) for every \(h\).  If \(\varpi\) belongs to the subdifferential of the finite convex function \(F_{\mathrm{hom}}\) at \(\xi\), then \(\varpi\cdot h\leq F_{\mathrm{hom}}'(\xi;h)\leq A_{\mathrm{hom}}(\xi)\cdot h\) for every \(h\), so \(\varpi=A_{\mathrm{hom}}(\xi)\).  A finite convex function on \(\mathbb{R}^{N}\) whose subdifferential is a singleton at every point is differentiable, and its gradient is continuous \cite[Section~25]{Rockafellar1970}.  This proves \(F_{\mathrm{hom}}\in C^{1}\) and \eqref{eq:Fhom-derivative}.

Extend \(A(y,G_\xi(y))\) by zero to \(T\).  Since \(|Y|=1\), Jensen's inequality, \eqref{eq:A-growth}, \eqref{eq:F-growth} and the upper bound \(F_{\mathrm{hom}}(\xi)\leq c_{1}|Y^*|\Phi(|\xi|)\) give
\[
\widetilde\Phi\!\left(\frac{|A_{\mathrm{hom}}(\xi)|}{c_2}\right)
\le\int_{Y^*}\widetilde\Phi\!\left(\frac{|A(y,G_\xi)|}{c_2}\right)dy
\le \int_{Y^*}\Phi(|G_\xi|)\,dy
\le\frac{1}{c_0}F_{\mathrm{hom}}(\xi),
\]
which is \eqref{eq:Ahom-complementary-growth}.  By Lemma~\ref{lem:nonlinear-modular-continuity}(b) with \(\mathcal{B}=A_{\mathrm{hom}}\), the map \(\nabla u\mapsto A_{\mathrm{hom}}(\nabla u)\) is continuous from \(L^{\Phi}(\Omega)^{N}\) into \(L^{\widetilde{\Phi}}(\Omega)^{N}\).  The fundamental theorem of calculus applied to \(t\mapsto F_{\mathrm{hom}}(\nabla u+t\nabla v)\) and H\"older's inequality \eqref{eq:Orlicz-Holder} give \eqref{eq:homogenized-energy-C1}.
\end{proof}

\begin{proposition}
\label{prop:homogenized-existence}
For every \(\alpha>0\), the minimum
\(\mu_{1,0}(\alpha)\) in \eqref{eq:intro-mu0} is attained by a nonnegative
\(u_{0,\alpha}\in W_{0}^{1}L^{\Phi}(\Omega)\). Every minimizer \(u\) has a unique associated positive multiplier \(\lambda_0[u]\) satisfying the homogenized variational equation \eqref{eq:intro-limit-equation} with \(u\) in place of \(u_{0,\alpha}\).
\end{proposition}

\begin{proof}
Apply Lemma~\ref{lem:constrained-minimization} with \(X=W_{0}^{1}L^{\Phi}(\Omega)\).  Sublevel sets of \(\mathcal{E}_{0}\) are bounded by \eqref{eq:Fhom-growth} and \eqref{eq:Orlicz-Poincare}; \(\mathcal{E}_{0}\) is \(C^{1}\) by \eqref{eq:homogenized-energy-C1}; \(\mathcal{M}_{0}\) is \(C^{1}\) and sequentially weakly continuous by \eqref{eq:compact-volume-embedding} and Lemma~\ref{lem:nonlinear-modular-continuity}; and \(\mathcal{M}_{0}'(v)[v]\geq p^{-}\mathcal{M}_{0}(v)\) by \eqref{eq:Phi-indices}.  Since \(F_{\mathrm{hom}}\) is convex with \(F_{\mathrm{hom}}(0)=0\),
\[
A_{\mathrm{hom}}(\xi)\cdot\xi
\geq F_{\mathrm{hom}}(\xi)
\geq c\Phi(\lvert\xi\rvert).
\]
If \(u\) is any minimizer, then \(\mathcal M_0(u)=\alpha>0\), so
\(u\neq0\), and the Poincar\'e inequality \eqref{eq:Orlicz-Poincare} implies
\(\nabla u\neq0\). Therefore
\(\mathcal E_0'(u)[u]>0\); since
\(\mathcal M_0'(u)[u]\geq p^-\alpha>0\), the multiplier
\(\lambda_0[u]\) in \eqref{eq:abstract-multiplier} is positive.
Evenness of \(F_{\mathrm{hom}}\) permits replacing a minimizer by its
absolute value.
\end{proof}

\subsection{Recovery sequence}
\label{subsec:recovery}

The upper bound for the constrained minima, that is, the upper-bound inequality of \(\Gamma\)-convergence, requires approximating every admissible macroscopic function by microscopic competitors whose energies converge to the homogenized energy and whose surface modulars converge to the limit modular.

\begin{proposition}
\label{prop:recovery}
For every \(v\in W_{0}^{1}L^{\Phi}(\Omega)\), there are \(v_{\varepsilon}\in W_{0}^{1}L^{\Phi}(\Omega)\), \(\varepsilon\in(0,1)\), bounded in \(W^{1}L^{\Phi}(\Omega)\) for small \(\varepsilon\), such that \(v_{\varepsilon}\to v\) in \(L^{\Phi}(\Omega)\), \(v_{\varepsilon}|_{\Omega_{\varepsilon}}\in X_{\varepsilon}\), and
\begin{equation}
\mathcal{M}_{\varepsilon}(v_{\varepsilon})
\longrightarrow
\mathcal{M}_{0}(v),
\qquad
\limsup_{\varepsilon\to0}
\mathcal{E}_{\varepsilon}(v_{\varepsilon})
\leq
\mathcal{E}_{0}(v).
\label{eq:recovery-energy}
\end{equation}
If \(\mathcal{M}_{0}(v)=\alpha>0\), the sequence can be chosen with
\(\mathcal{M}_{\varepsilon}(v_{\varepsilon})=\alpha\)
for every sufficiently small \(\varepsilon\).
\end{proposition}

\begin{proof}
Recovery sequences for convex periodic integral functionals without
perforation or constraint are constructed in
\cite{TachagoGargiuloNnangZappale2021,TachagoNnang2012}; the perforated
domain and the exact surface constraint require the following construction.

For \(\xi\in\mathbb{R}^{N}\), extend \(w_{\xi}\) to \(Y\) by Lemma~\ref{lem:cell-extension-trace}; since \(Pw_{\xi}=w_{\xi}\) near \(\partial Y\), its periodic extension belongs to \(W^{1}L^{\Phi}_{\mathrm{loc}}(\mathbb{R}^{N})\).  Periodic mollification and translation continuity under \(\Delta_{2}\) give \(w_{\xi,n}\in C_{\mathrm{per}}^{\infty}(Y)\) with \(w_{\xi,n}\to Pw_{\xi}\) in \(W^{1}L^{\Phi}(Y)\).  By \eqref{eq:cell-extension-one}, Lemma~\ref{lem:affine-periodic-coercivity}, \eqref{eq:F-growth} and \eqref{eq:Fhom-growth},
\begin{equation}
\int_{Y}\Phi\bigl(\lvert\xi+\nabla_{y}Pw_{\xi}\rvert\bigr)\,dy
\leq
C\Phi(\lvert\xi\rvert).
\label{eq:corrector-full-bound}
\end{equation}

Let \(Q\Subset\Omega\) be an open cube and suppose \(v(x)=\xi\cdot x+c\) in \(Q\).
Fix \(n\in\mathbb{N}\) temporarily and define
\[
v_{\varepsilon,n}(x)
=
v(x)
+
\varepsilon
w_{\xi,n}\left(\frac{x}{\varepsilon}\right).
\]
For every \(i=1,\ldots,N\),
\[
D_{x_{i}}v_{\varepsilon,n}
=
\xi_{i}
+
D_{y_{i}}w_{\xi,n}\left(\frac{x}{\varepsilon}\right).
\]
Let \(Q_{\varepsilon}^{\mathrm{int}}\) be the union of the complete cells
\(\varepsilon(k+Y)\) contained in \(Q\).  Cell periodicity gives the exact identity
\begin{equation}
\int_{Q_{\varepsilon}^{\mathrm{int}}\cap\Omega_{\varepsilon}}
F\left(
\frac{x}{\varepsilon},
\nabla v_{\varepsilon,n}(x)
\right)\,dx
=
\lvert Q_{\varepsilon}^{\mathrm{int}}\rvert
\int_{Y^{*}}
F\bigl(y,\xi+\nabla_{y}w_{\xi,n}(y)\bigr)\,dy.
\label{eq:affine-recovery-identity}
\end{equation}
Because \(w_{\xi,n}\) is smooth and fixed,
\[
\left|
Q\setminus Q_{\varepsilon}^{\mathrm{int}}
\right|
\longrightarrow0
\]
and \eqref{eq:F-growth} control the incomplete cells.  Thus
\begin{equation}
\lim_{\varepsilon\to0}
\int_{Q\cap\Omega_{\varepsilon}}
F\left(
\frac{x}{\varepsilon},
\nabla v_{\varepsilon,n}
\right)\,dx
=
\lvert Q\rvert
\int_{Y^{*}}
F\bigl(y,\xi+\nabla_{y}w_{\xi,n}\bigr)\,dy.
\label{eq:affine-recovery-limit}
\end{equation}
The same argument applies to every bounded set
\(D\Subset\Omega\) with \(|\partial D|=0\): the union of the
\(\varepsilon\)-cells meeting \(\partial D\) has measure tending to zero.
For the polyhedral partition below, the parameters \(\delta\) can and will
be chosen so that each set \(Q_j\setminus S_\delta\) has boundary of
Lebesgue measure zero. Thus \eqref{eq:affine-recovery-limit} applies on
all the regions used in the gluing argument.
Letting \(n\to\infty\), the continuity and growth of \(F\), together with modular convergence, give the cell minimum
\[
\int_{Y^{*}}
F\bigl(y,\xi+\nabla_{y}w_{\xi}\bigr)\,dy
=
F_{\mathrm{hom}}(\xi).
\]

We next glue the construction for continuous piecewise-affine functions.
Assume that \(v\) is continuous and affine on each member of a finite polyhedral partition
\[
\{Q_{j}:j=1,\ldots,J\}
\]
of its support, where \(J\in\mathbb{N}\), and that
\[
\operatorname{dist}(\operatorname{supp}v,\partial\Omega)>0.
\]
Write
\[
\nabla v=\xi_{j}
\quad\text{in }Q_{j}.
\]
For a number \(\delta>0\), let \(S_{\delta}\) be the
\(2\delta\)-neighbourhood of all interfaces
\(\partial Q_{j}\) and of \(\partial\Omega\).  Choose smooth functions
\(\eta_{\delta,j}\) such that
\begin{align*}
0\leq\eta_{\delta,j}\leq1,
\qquad
\eta_{\delta,j}&=1
\quad\text{on }
Q_{j}\setminus S_{\delta},
\\
\operatorname{supp}\eta_{\delta,j}
&\subset Q_{j},
\\
\lvert\nabla\eta_{\delta,j}\rvert
&\leq\frac{C}{\delta}.
\end{align*}
For fixed smooth correctors \(w_{\xi_{j},n}\), define on \(\Omega\)
\begin{equation}
v_{\varepsilon,\delta,n}(x)
=
v(x)
+
\varepsilon
\sum_{j=1}^{J}
\eta_{\delta,j}(x)
w_{\xi_{j},n}\left(\frac{x}{\varepsilon}\right).
\label{eq:piecewise-recovery}
\end{equation}
The function vanishes near \(\partial\Omega\); its restriction to
\(\Omega_{\varepsilon}\) belongs to \(X_{\varepsilon}\).  Its gradient is
\begin{equation}
\nabla v_{\varepsilon,\delta,n}
=
\nabla v
+
\sum_{j=1}^{J}
\eta_{\delta,j}
\nabla_{y}w_{\xi_{j},n}\left(\frac{x}{\varepsilon}\right)
\\
+
\varepsilon
\sum_{j=1}^{J}
w_{\xi_{j},n}\left(\frac{x}{\varepsilon}\right)
\nabla\eta_{\delta,j}.
\label{eq:piecewise-recovery-gradient}
\end{equation}
For fixed \(\delta\) and \(n\), the last sum tends uniformly to zero as
\(\varepsilon\to0\).  Outside \(S_{\delta}\), formula
\eqref{eq:affine-recovery-limit} applies on every
\(Q_{j}\setminus S_{\delta}\).  Inside \(S_{\delta}\), all correctors are smooth and bounded, and
\eqref{eq:F-growth} gives
\[
\limsup_{\varepsilon\to0}
\int_{S_{\delta}\cap\Omega_{\varepsilon}}
F\left(
\frac{x}{\varepsilon},
\nabla v_{\varepsilon,\delta,n}
\right)\,dx
\leq
C_{v,n}\lvert S_{\delta}\rvert.
\]
Since the partition is finite,
\(\lvert S_{\delta}\rvert\to0\) as \(\delta\to0\).  Consequently,
\begin{equation}
\limsup_{n\to\infty}
\limsup_{\delta\to0}
\limsup_{\varepsilon\to0}
\mathcal{E}_{\varepsilon}
\left(
v_{\varepsilon,\delta,n}|_{\Omega_{\varepsilon}}
\right)
\leq
\sum_{j=1}^{J}
\lvert Q_{j}\rvert F_{\mathrm{hom}}(\xi_{j})
=
\int_{\Omega}F_{\mathrm{hom}}(\nabla v)\,dx.
\label{eq:piecewise-energy-upper}
\end{equation}

The functions \(v_{\varepsilon,\delta,n}\) belong to \(W_{0}^{1}L^{\Phi}(\Omega)\) and converge to \(v\) in \(L^{\Phi}(\Omega)\) as \(\varepsilon\to0\), since the perturbation has amplitude \(O(\varepsilon)\).  Writing \(\xi_{j}+\eta_{\delta,j}\nabla_{y}w=(1-\eta_{\delta,j})\xi_{j}+\eta_{\delta,j}(\xi_{j}+\nabla_{y}w)\), convexity, periodicity and \eqref{eq:corrector-full-bound} give
\begin{equation}
\limsup_{n\to\infty}\limsup_{\delta\to0}\limsup_{\varepsilon\to0}
\int_{\Omega}\Phi\bigl(\lvert\nabla v_{\varepsilon,\delta,n}\rvert\bigr)\,dx
\leq
C\int_{\Omega}\Phi(\lvert\nabla v\rvert)\,dx.
\label{eq:recovery-full-gradient}
\end{equation}

For general \(v\in W_{0}^{1}L^{\Phi}(\Omega)\), choose continuous
piecewise-affine \(v^m\) with compact support such that
\(v^m\to v\) in \(W_0^1L^\Phi(\Omega)\) (smooth density
\cite{Gossez1982} and piecewise-affine interpolation). By the growth
of \(F_{\mathrm{hom}}\) and Orlicz modular continuity,
\(\mathcal E_0(v^m)\to\mathcal E_0(v)\). Write \(\xi_j^m\) for the
gradients of \(v^m\) on the partition \(\{Q_j^m:j=1,\ldots,J_m\}\), and put
\(B_m=\int_\Omega\Phi(|\nabla v^m|)\,dx\). In view of the strong
periodic corrector approximation and \eqref{eq:corrector-full-bound},
for each \(m\) we can first fix \(n(m)\geq m\) so that
\begin{align*}
\sum_{j=1}^{J_m}|Q_j^m|\int_{Y^*}F(y,\xi_j^m+\nabla_yw_{\xi_j^m,n(m)})\,dy
&\leq\mathcal E_0(v^m)+\frac{1}{4m},\\
\sum_{j=1}^{J_m}|Q_j^m|\int_Y\Phi(|\xi_j^m+\nabla_yw_{\xi_j^m,n(m)}|)\,dy
&\leq C_0B_m+\frac{1}{4m},
\end{align*}
where \(C_0\) depends only on the structural constants, not on \(m\).  Indeed, for each fixed \(\xi\), the convergence \(w_{\xi,n}\to Pw_{\xi}\) in \(W^{1}L^{\Phi}(Y)\) and Lemma~\ref{lem:nonlinear-modular-continuity} give \(\int_Y\Phi(|\xi+\nabla_yw_{\xi,n}|)\,dy\to\int_Y\Phi(|\xi+\nabla_yPw_{\xi}|)\,dy\leq C\Phi(|\xi|)\) by \eqref{eq:corrector-full-bound}; summing over the finitely many \(\xi_j^m\) gives the second inequality with \(C_0=C\), the constant of \eqref{eq:corrector-full-bound}, and the first one follows in the same way from the cell minimality of \(w_{\xi_j^m}\).
With these finitely many smooth correctors fixed, the cutoff-gradient
term satisfies, for every fixed \(m,n,\delta\),
\[
\left\|\varepsilon\sum_{j=1}^{J_m}
w_{\xi_j^m,n}(\cdot/\varepsilon)\nabla\eta_{\delta,j}
\right\|_{L^\infty(\Omega)}
\leq\varepsilon C_{m,n,\delta}\longrightarrow0.
\]
The other terms in \eqref{eq:piecewise-recovery-gradient} are uniformly
bounded when \(m,n\) are fixed. By the growth bounds and continuity of
\(F\) and \(\Phi\) on bounded gradient ranges, this perturbation
changes both integrals by \(o_{\varepsilon\to0}(1)\) with
\(m,n,\delta\) fixed. Periodic averaging off the interface strips,
convexity for \(0\leq\eta_{\delta,j}\leq1\), and the bound on the strips
give
\begin{align*}
\limsup_{\varepsilon\to0}\mathcal E_\varepsilon
 (v^m_{\varepsilon,\delta,n})
&\leq\sum_{j=1}^{J_m}|Q_j^m|
 \int_{Y^*}F(y,\xi_j^m+\nabla_yw_{\xi_j^m,n})\,dy
 +C_{m,n}|S_\delta|,\\
\limsup_{\varepsilon\to0}\int_\Omega
 \Phi(|\nabla v^m_{\varepsilon,\delta,n}|)\,dx
&\leq\sum_{j=1}^{J_m}|Q_j^m|
 \int_Y\Phi(|\xi_j^m+\nabla_yw_{\xi_j^m,n}|)\,dy
 +C_{m,n}|S_\delta|.
\end{align*}
Here \(C_{m,n}\) is independent of \(\delta\); indeed the strip
integrands without the vanishing cutoff-gradient term are bounded by
the fixed smooth correctors and the finitely many gradients \(\xi_j^m\).
After fixing \(n=n(m)\), choose \(0<\delta(m)<1/m\) so that both strip
errors are below \(1/(4m)\). Finally choose
\(0<s_m<\min\{s_{m-1}/2,1/m\}\), with \(s_0=1\), so that, for every
\(0<\varepsilon<s_m\), the periodic averaging errors, the
cutoff-gradient error and the distance to \(v^m\) in \(L^\Phi(\Omega)\) are at
most \(1/(4m)\). Such a threshold exists for each fixed
\(m,n(m),\delta(m)\), and the constant \(C_{m,n(m),\delta(m)}\) may
depend on \(m\). Thus, putting
\(v_\varepsilon^m=v^m_{\varepsilon,\delta(m),n(m)}\), the same
\(s_m\) ensures all three inequalities
\begin{equation}
\begin{aligned}
\mathcal E_\varepsilon(v_\varepsilon^m)
&\leq\mathcal E_0(v^m)+\frac1m,\\
\|v_\varepsilon^m-v^m\|_{\Phi,\Omega}&\leq\frac1m,\\
\int_\Omega\Phi(|\nabla v_\varepsilon^m|)\,dx
&\leq C_0B_m+1.
\end{aligned}
\label{eq:recovery-diagonal-estimates}
\end{equation}
The constant \(C_0\) is independent of all diagonal choices:
it comes from the cell bound before \(n(m)\) and \(\delta(m)\) are
selected. Since \(v^m\to v\) in \(W_0^1L^\Phi(\Omega)\), the sequence
\((B_m)\) is bounded by \(\Delta_2\). For
\(s_{m+1}\leq\varepsilon<s_{m}\), set
\(v_\varepsilon=v_\varepsilon^m\). Then
\(v_\varepsilon\to v\) in \(L^\Phi(\Omega)\), and
\eqref{eq:recovery-diagonal-estimates}, with the Poincar\'e
inequality \eqref{eq:Orlicz-Poincare}, gives a bound in \(W_0^1L^\Phi(\Omega)\) independent of
\(\varepsilon\). Consequently \(v_\varepsilon\rightharpoonup v\)
in that space, and
\(\limsup_{\varepsilon\to0}\mathcal E_\varepsilon(v_\varepsilon)
\leq\mathcal E_0(v)\).
  Proposition~\ref{prop:nonlinear-surface-compactness}, applied with \(U_{\varepsilon}=v_{\varepsilon}\), gives \(\mathcal{M}_{\varepsilon}(v_{\varepsilon})\to\mathcal{M}_{0}(v)\).  For some \(\varepsilon_0>0\), the energy bound also gives
\begin{equation}
\sup_{0<\varepsilon<\varepsilon_0}
\int_{\Omega_{\varepsilon}}\Phi(\lvert\nabla v_{\varepsilon}\rvert)\,dx
<\infty.
\label{eq:recovery-uniform-gradient}
\end{equation}

If \(\mathcal{M}_{0}(v)=\alpha>0\), let \(t_{\varepsilon}\) satisfy \(\mathcal{M}_{\varepsilon}(t_{\varepsilon}v_{\varepsilon})=\alpha\).  Since \(\mathcal{M}_{\varepsilon}(v_{\varepsilon})\to\alpha\), \eqref{eq:normalization-scaling} gives
\begin{equation}
t_{\varepsilon}\longrightarrow1.
\label{eq:normalization-factor-one}
\end{equation}
It remains to verify that the energy upper bound is preserved.  Put
\(\xi_{\varepsilon}=\nabla v_{\varepsilon}\) and
\(\theta_{\varepsilon,s}=1+s(t_{\varepsilon}-1)\).  The fundamental theorem of calculus gives, for almost every \(x\in\Omega_{\varepsilon}\),
\begin{equation}
\left|F\!\left(\frac{x}{\varepsilon},t_{\varepsilon}\xi_{\varepsilon}\right)-F\!\left(\frac{x}{\varepsilon},\xi_{\varepsilon}\right)\right|
\leq |t_{\varepsilon}-1|\,|\xi_{\varepsilon}|\int_{0}^{1}\left|A\!\left(\frac{x}{\varepsilon},\theta_{\varepsilon,s}\xi_{\varepsilon}\right)\right|\,ds.
\label{eq:normalization-energy-difference}
\end{equation}
By \eqref{eq:normalization-factor-one}, \(\theta_{\varepsilon,s}\in[1/2,2]\) for all small \(\varepsilon\) and every \(s\in[0,1]\).  Young's inequality, \eqref{eq:A-growth}, and \eqref{eq:fixed-scaling} then give, uniformly for \(\theta\in[1/2,2]\),
\[
|\xi|\,|A(y,\theta\xi)|
\leq \Phi(c_{2}|\xi|)+\widetilde\Phi\!\left(\frac{|A(y,\theta\xi)|}{c_{2}}\right)
\leq C\Phi(|\xi|).
\]
Integrating \eqref{eq:normalization-energy-difference} and using
\eqref{eq:recovery-uniform-gradient} therefore yields
\begin{equation}
\left|
\mathcal{E}_{\varepsilon}(t_{\varepsilon}v_{\varepsilon})
-
\mathcal{E}_{\varepsilon}(v_{\varepsilon})
\right|
\leq
C\lvert t_{\varepsilon}-1\rvert
\int_{\Omega_{\varepsilon}}
\Phi(\lvert\nabla v_{\varepsilon}\rvert)\,dx
\longrightarrow0.
\label{eq:normalized-recovery-energy-stability}
\end{equation}
Thus \eqref{eq:recovery-energy} is preserved, and replacing \(v_{\varepsilon}\) by \(t_{\varepsilon}v_{\varepsilon}\) completes the proof.
\end{proof}

\subsection{Passage to the limit and proof of the main results}
\label{subsec:proof-main}

The uniform estimates of Subsection~\ref{subsec:estimates}, the homogenized problem of Subsection~\ref{subsec:macroscopic} and the recovery sequence of Subsection~\ref{subsec:recovery} are now combined.

\begin{proof}[Proof of Theorems~\ref{thm:intro-main} and~\ref{thm:intro-full}]
Start from an arbitrary sequence \(\varepsilon\to0\) in \(E\), which is
not relabelled. By Theorem~\ref{thm:uniform-estimates},
\(u_{\varepsilon,\alpha}\) is bounded in
\(W^{1}L^{\Phi}(\Omega_{\varepsilon})\), and
Proposition~\ref{prop:perforated-gradient-compactness} gives, after
extraction,
\(u_{0,\alpha}\in W_{0}^{1}L^{\Phi}(\Omega)\) and
\(u_{1,\alpha}\in L^{1}(\Omega;W_{\#}^{1}L^{\Phi}(Y^{*}))\), with
\(\nabla_{y}u_{1,\alpha}\in L^{\Phi}(\Omega\times Y^{*})^{N}\), such that
\eqref{eq:intro-weak-u}--\eqref{eq:intro-gradient-convergence} hold.  Proposition~\ref{prop:nonlinear-surface-compactness} gives
\[
\alpha
=
\lim_{\varepsilon\to0}
\mathcal{M}_{\varepsilon}(u_{\varepsilon,\alpha})
=
\beta_{\rho}
\int_{\Omega}\Phi(\lvert u_{0,\alpha}\rvert)\,dx.
\]
Thus \(u_{0,\alpha}\) is admissible for the limit minimum.

The limit \(u_{0,\alpha}\) is nonnegative.  Indeed, put \(U_{\varepsilon}=P_{\varepsilon}u_{\varepsilon,\alpha}\); since \(u_{\varepsilon,\alpha}\geq0\), we have \(\chi_{\varepsilon}U_{\varepsilon}^{-}=0\), where \(t^{-}=\max\{-t,0\}\).  As \(t\mapsto t^{-}\) is \(1\)-Lipschitz, \eqref{eq:intro-strong-u} gives \(U_{\varepsilon}^{-}\to u_{0,\alpha}^{-}\) in \(L^{1}(\Omega)\), hence \(\int_{\Omega}\chi_{\varepsilon}u_{0,\alpha}^{-}\,dx\to0\).  The periodic mean-value property \cite[Chapter~2]{CioranescuDonato1999}, applied to \(\chi_{\varepsilon}\) on complete cells, and density in \(L^{1}(\Omega)\) show that this integral also converges to \(|Y^{*}|\int_{\Omega}u_{0,\alpha}^{-}\,dx\); thus \(u_{0,\alpha}\geq0\) almost everywhere.

The multiplier bound \eqref{eq:uniform-lambda} yields, after a further extraction,
\begin{equation}
\lambda_{1,\varepsilon}(\alpha)
\longrightarrow
\lambda_{*}
\quad\text{for some }\lambda_{*}>0.
\label{eq:lambda-subsequence}
\end{equation}

We first establish the two-scale lower bound along this arbitrary
extraction.
Set
\begin{equation}
\Xi(x,y)
=
\nabla u_{0,\alpha}(x)
+\nabla_{y}u_{1,\alpha}(x,y)
\quad\text{in }\Omega\times Y^{*}.
\label{eq:limit-gradient-G}
\end{equation}
Let
\[
\Psi=(\Psi_{i})_{i=1}^{N}
\in
C_{c}^{\infty}
\bigl(\Omega;C_{\mathrm{per}}^{\infty}(Y)^{N}\bigr).
\]
Convexity of \(F(y,\cdot)\) gives, for almost every
\(x\in\Omega_{\varepsilon}\),
\begin{equation}
F\left(
\frac{x}{\varepsilon},
\nabla u_{\varepsilon,\alpha}
\right)
\geq
F\left(
\frac{x}{\varepsilon},
\Psi\left(x,\frac{x}{\varepsilon}\right)
\right)
+
\sum_{i=1}^{N}
A_{i}\left(
\frac{x}{\varepsilon},
\Psi\left(x,\frac{x}{\varepsilon}\right)
\right)
\left[
D_{x_{i}}u_{\varepsilon,\alpha}
-
\Psi_{i}\left(x,\frac{x}{\varepsilon}\right)
\right].
\label{eq:convex-supporting-inequality}
\end{equation}
Integrate over \(\Omega_{\varepsilon}\).  The function
\(x\mapsto\chi_{Y^{*}}(\cdot)F(\cdot,\Psi(x,\cdot))\) is continuous from \(\overline{\Omega}\) into the bounded periodic functions, and
\(\chi_{\varepsilon}\) differs from \(\chi_{Y^{*}}(x/\varepsilon)\) only on \(R_{\varepsilon}\); the periodic mean-value property
\cite[Chapter~2]{CioranescuDonato1999} therefore gives
\[
\int_{\Omega}\chi_{\varepsilon}F\Bigl(\frac{x}{\varepsilon},\Psi\Bigl(x,\frac{x}{\varepsilon}\Bigr)\Bigr)\,dx\longrightarrow\int_{\Omega}\int_{Y^{*}}F(y,\Psi(x,y))\,dy\,dx,
\]
and similarly for the term \(A\cdot\Psi\).  Since \(A_{i}(y,\Psi(x,y))\in C_{c}(\Omega;C_{\mathrm{per}}(Y))\) is an element of \(L^{\widetilde{\Phi}}(\Omega;C_{\mathrm{per}}(Y))\), hence a test function for \eqref{eq:intro-gradient-convergence}, passing to the limit gives
\begin{multline}
\liminf_{\varepsilon\to0}
\mathcal{E}_{\varepsilon}(u_{\varepsilon,\alpha})
\geq
\int_{\Omega}\int_{Y^{*}}
\biggl[
F(y,\Psi(x,y))
\\
+
\sum_{i=1}^{N}
A_{i}(y,\Psi(x,y))
\bigl(\Xi_{i}(x,y)-\Psi_{i}(x,y)\bigr)
\biggr]\,dy\,dx.\qquad\qquad
\label{eq:preliminary-liminf}
\end{multline}

By Lemma~\ref{lem:smooth-periodic-density}, choose a sequence
\[
(\Psi_{n})_{n\in\mathbb{N}}
\subset
C_{c}^{\infty}\bigl(\Omega;C_{\mathrm{per}}^{\infty}(Y)^{N}\bigr)
\]
such that
\begin{equation}
\Psi_{n}\longrightarrow \Xi
\quad\text{in }L^{\Phi}(\Omega\times Y^{*})^{N}.
\label{eq:liminf-smooth-approximation}
\end{equation}
Lemma~\ref{lem:nonlinear-modular-continuity} gives
\begin{align*}
F(y,\Psi_{n})&\longrightarrow F(y,\Xi)
&&\text{in }L^{1}(\Omega\times Y^{*}),
\\
A(y,\Psi_{n})&\longrightarrow A(y,\Xi)
&&\text{in }L^{\widetilde\Phi}(\Omega\times Y^{*})^{N}.
\end{align*}
In particular, \(A(y,\Psi_{n})\) is bounded in
\(L^{\widetilde\Phi}(\Omega\times Y^{*})^{N}\), and H\"older's inequality \eqref{eq:Orlicz-Holder} gives
\[
\left|
\int_{\Omega}\int_{Y^{*}}
A(y,\Psi_{n})\cdot(\Xi-\Psi_{n})\,dy\,dx
\right|
\leq
C
\|A(y,\Psi_{n})\|_{\widetilde\Phi,\Omega\times Y^{*}}
\|\Xi-\Psi_{n}\|_{\Phi,\Omega\times Y^{*}}
\longrightarrow0.
\]
We may therefore let \(n\to\infty\) in
\eqref{eq:preliminary-liminf}.  Using the cell definition
\eqref{eq:intro-Fhom} with \(w=u_{1,\alpha}(x,\cdot)\), for almost every \(x\in\Omega\), we obtain
\begin{equation}
\liminf_{\varepsilon\to0}\mathcal E_{\varepsilon}(u_{\varepsilon,\alpha})
\geq \int_{\Omega}\!\int_{Y^{*}}F(y,\Xi(x,y))\,dy\,dx
\geq \mathcal E_{0}(u_{0,\alpha}).
\label{eq:main-liminf}
\end{equation}
This is the lower-bound inequality of \(\Gamma\)-convergence.

We now determine the limit of the minima independently of the preceding
arbitrary extraction. Choose a sequence \(\varepsilon_j\to0\) realizing
\(\liminf_{\varepsilon\to0}\mu_{1,\varepsilon}(\alpha)\), and apply the
compactness and lower-bound argument above to a subsequence. Denoting its
normalized limit by \(\overline u_{0,\alpha}\), we obtain
\[
\liminf_{\varepsilon\to0}\mu_{1,\varepsilon}(\alpha)
\geq
\mathcal E_0(\overline u_{0,\alpha})
\geq
\mu_{1,0}(\alpha).
\]
For the complementary upper bound, let
\(v\in W_{0}^{1}L^{\Phi}(\Omega)\) satisfy
\(\mathcal M_{0}(v)=\alpha\). Proposition~\ref{prop:recovery} provides an
exactly normalized recovery sequence and hence, by minimality,
\[
\limsup_{\varepsilon\to0}\mu_{1,\varepsilon}(\alpha)
\leq \mathcal E_{0}(v).
\]
Taking the infimum over normalized \(v\) gives
\[
\mu_{1,0}(\alpha)
\leq \liminf_{\varepsilon\to0}\mu_{1,\varepsilon}(\alpha)
\leq \limsup_{\varepsilon\to0}\mu_{1,\varepsilon}(\alpha)
\leq \mu_{1,0}(\alpha).
\]
Thus all inequalities are equalities,
\(\mu_{1,\varepsilon}(\alpha)\to\mu_{1,0}(\alpha)\) along the full
family.

Return now to the arbitrary sequence and the extracted fields
\(u_{0,\alpha}\), \(u_{1,\alpha}\), and \(\lambda_*\) fixed at the
beginning of the proof. Since
\(\mathcal E_\varepsilon(u_{\varepsilon,\alpha})
=\mu_{1,\varepsilon}(\alpha)\to\mu_{1,0}(\alpha)\),
\eqref{eq:main-liminf} and the admissibility of \(u_{0,\alpha}\) give
\[
\mu_{1,0}(\alpha)
\geq
\int_\Omega\!\int_{Y^*}F(y,\Xi)\,dy\,dx
\geq
\mathcal E_0(u_{0,\alpha})
\geq
\mu_{1,0}(\alpha).
\]
Hence all three quantities are equal and \(u_{0,\alpha}\) is a minimizer
of the homogenized problem. The remainder of the proof therefore applies
to an arbitrary initial sequence.

We then extract a two-scale flux limit and pass to both variational equations.
For every \(i=1,\ldots,N\), define the zero extension
\begin{equation}
q_{\varepsilon,i}(x)
=
\chi_{\varepsilon}(x)
A_{i}\left(
\frac{x}{\varepsilon},
\nabla u_{\varepsilon,\alpha}(x)
\right).
\label{eq:micro-flux}
\end{equation}
The growth \eqref{eq:A-growth} and the gradient estimate
\eqref{eq:uniform-gradient} bound
\(q_{\varepsilon,i}\) in
\(L^{\widetilde{\Phi}}(\Omega)\).  Applying
Theorem~\ref{thm:volume-two-scale-compactness} with
\(\widetilde{\Phi}\) in place of \(\Phi\), we first obtain, after a
further extraction,
\(\widehat q_i\in L^{\widetilde\Phi}(\Omega\times Y)\) such that
\(q_{\varepsilon,i}\overset{2s}{\rightharpoonup}\widehat q_i\).
Since
\(q_{\varepsilon,i}=\chi_\varepsilon q_{\varepsilon,i}\),
Lemma~\ref{lem:solid-phase-restriction} applied to this same sequence gives
\(q_{\varepsilon,i}\overset{2s}{\rightharpoonup}
\chi_{Y^*}\widehat q_i\). Uniqueness of the two-scale limit therefore
implies \(\widehat q_i=\chi_{Y^*}\widehat q_i\). Define
\(q_i=\widehat q_i|_{\Omega\times Y^*}\). Then
\(q_i\in L^{\widetilde\Phi}(\Omega\times Y^*)\), and
\begin{equation}
q_{\varepsilon,i}
\overset{2s}{\rightharpoonup}
\chi_{Y^{*}}q_{i}
\quad\text{for every }i=1,\ldots,N.
\label{eq:unidentified-flux-limit}
\end{equation}

Take \(v\in C_{c}^{\infty}(\Omega)\) in the variational equation \eqref{eq:intro-weak}.  The volume side converges by
\eqref{eq:unidentified-flux-limit}.  The surface side converges by
\eqref{eq:lambda-subsequence} and
\eqref{eq:surface-derivative-limit}.  Hence, for every \(v\in C_{c}^{\infty}(\Omega)\),
\begin{equation}
\int_{\Omega}\int_{Y^{*}}
\sum_{i=1}^{N}
q_{i}(x,y)D_{x_{i}}v(x)\,dy\,dx
=
\lambda_{*}\beta_{\rho}
\int_{\Omega}
\phi(\lvert u_{0,\alpha}\rvert)
\operatorname{sgn}(u_{0,\alpha})v\,dx.
\label{eq:unidentified-macro-equation}
\end{equation}

Next take
\[
f\in C_{c}^{\infty}(\Omega),
\qquad
\zeta\in C_{\mathrm{per}}^{\infty}(Y),
\]
and use
\[
v_{\varepsilon}(x)
=
\varepsilon
f(x)
\zeta\left(\frac{x}{\varepsilon}\right)
\]
as a microscopic test.  For every \(i=1,\ldots,N\),
\[
D_{x_{i}}v_{\varepsilon}
=
\varepsilon
D_{x_{i}}f(x)
\zeta\left(\frac{x}{\varepsilon}\right)
+
f(x)
D_{y_{i}}\zeta\left(\frac{x}{\varepsilon}\right).
\]
The boundary term tends to zero.  Indeed, after writing it with
\(d\vartheta_{\varepsilon}\), it contains an additional factor
\(\varepsilon\), while
\eqref{eq:uniform-boundary-dual} and H\"older's inequality \eqref{eq:Orlicz-Holder} give a uniform bound for the remaining product.  Passing to the limit in the volume term yields, for every \(f\in C_{c}^{\infty}(\Omega)\) and every \(\zeta\in C_{\mathrm{per}}^{\infty}(Y)\),
\begin{equation}
\int_{\Omega}\int_{Y^{*}}
\sum_{i=1}^{N}
q_{i}(x,y)
f(x)D_{y_{i}}\zeta(y)\,dy\,dx
=0.
\label{eq:unidentified-cell-equation}
\end{equation}
The restrictions of smooth periodic functions are dense in \(W_{\#}^{1}L^{\Phi}(Y^*)\) in the gradient norm, after subtracting their cell means: extend an element to \(Y\) using Lemma~\ref{lem:cell-extension-trace}, mollify periodically in \(W^1L^\Phi(Y)\), restrict to \(Y^*\), and subtract its mean on \(Y^*\). By separability choose a countable dense family \((\zeta_j)_{j\geq1}\) of these smooth restrictions. For each \(j\), the fundamental lemma in \(x\) applied to \eqref{eq:unidentified-cell-equation} shows that
\[
\int_{Y^*}q(x,y)\cdot\nabla_y\zeta_j(y)\,dy=0
\quad\text{for almost every }x\in\Omega.
\]
Intersect these countably many full-measure sets with the set on which \(q(x,\cdot)\in L^{\widetilde\Phi}(Y^*)^N\). For every \(x\) in the resulting full-measure set, H\"older's inequality \eqref{eq:Orlicz-Holder} and density yield, for every \(\zeta\in W_{\#}^{1}L^{\Phi}(Y^*)\),
\begin{equation}
\int_{Y^*}q(x,y)\cdot\nabla_y\zeta(y)\,dy=0.
\label{eq:pointwise-unidentified-cell}
\end{equation}
The exceptional set is therefore independent of \(\zeta\).

The variational equation \eqref{eq:intro-weak} yields convergence of the flux--gradient duality product.
By \eqref{eq:intro-multiplier-ratio}, the left-hand side of \eqref{eq:duality-product-limit} below equals \(\lambda_{1,\varepsilon}(\alpha)\,\varepsilon\int_{\Gamma_{\varepsilon}}\rho(x/\varepsilon)\phi(\lvert u_{\varepsilon,\alpha}\rvert)\lvert u_{\varepsilon,\alpha}\rvert\,d\sigma_{x}\).  By \eqref{eq:lambda-subsequence} and
\eqref{eq:surface-H-limit}, this quantity converges to
\begin{equation}
\lambda_{*}\beta_{\rho}
\int_{\Omega}
\phi(\lvert u_{0,\alpha}\rvert)
\lvert u_{0,\alpha}\rvert\,dx.
\label{eq:duality-boundary-limit}
\end{equation}
Equation \eqref{eq:unidentified-macro-equation} extends from smooth functions to
\(W_{0}^{1}L^{\Phi}(\Omega)\), because
\[
\int_{Y^{*}}q(x,y)\,dy
\in L^{\widetilde{\Phi}}(\Omega)^{N}
\]
and
\[
\phi(\lvert u_{0,\alpha}\rvert)
\operatorname{sgn}(u_{0,\alpha})
\in L^{\widetilde{\Phi}}(\Omega).
\]
Using \(v=u_{0,\alpha}\) shows that
\eqref{eq:duality-boundary-limit} equals
\[
\int_{\Omega}\int_{Y^{*}}
\sum_{i=1}^{N}
q_{i}(x,y)D_{x_{i}}u_{0,\alpha}(x)\,dy\,dx.
\]
For almost every \(x\), \(u_{1,\alpha}(x,\cdot)\in W_{\#}^{1}L^{\Phi}(Y^*)\), so \eqref{eq:pointwise-unidentified-cell} applies to this choice. The resulting integrand is in \(L^1(\Omega\times Y^*)\) by H\"older's inequality \eqref{eq:Orlicz-Holder} and the two-scale bounds on \(q\) and \(\nabla_yu_{1,\alpha}\). Integrating over \(x\) gives
\[
\int_{\Omega}\int_{Y^{*}}
\sum_{i=1}^{N}
q_{i}(x,y)D_{y_{i}}u_{1,\alpha}(x,y)\,dy\,dx
=0.
\]
Combining the last three identities gives
\begin{equation}
\lim_{\varepsilon\to0}
\int_{\Omega_{\varepsilon}}
\sum_{i=1}^{N}
A_{i}\left(
\frac{x}{\varepsilon},
\nabla u_{\varepsilon,\alpha}
\right)
D_{x_{i}}u_{\varepsilon,\alpha}\,dx
=
\int_{\Omega}\int_{Y^{*}}
\sum_{i=1}^{N}
q_{i}(x,y)\Xi_{i}(x,y)\,dy\,dx.
\label{eq:duality-product-limit}
\end{equation}

The Minty argument identifies the limiting nonlinear flux.
For every \(\Psi\in C_{c}^{\infty}(\Omega;C_{\mathrm{per}}^{\infty}(Y)^{N})\), monotonicity gives
\begin{equation}
\int_{\Omega_{\varepsilon}}
\sum_{i=1}^{N}
\left[
A_{i}\left(
\frac{x}{\varepsilon},
\nabla u_{\varepsilon,\alpha}
\right)
-
A_{i}\left(
\frac{x}{\varepsilon},
\Psi\left(x,\frac{x}{\varepsilon}\right)
\right)
\right]
\times
\left[
D_{x_{i}}u_{\varepsilon,\alpha}
-
\Psi_{i}\left(x,\frac{x}{\varepsilon}\right)
\right]\,dx
\geq0.
\label{eq:micro-Minty}
\end{equation}
Use \eqref{eq:duality-product-limit} for the only product of two weakly convergent sequences.  All other terms pass to the two-scale limit, since their oscillating factors are elements of \(L^{\widetilde{\Phi}}(\Omega;C_{\mathrm{per}}(Y))\) or of \(C_{c}(\Omega;C_{\mathrm{per}}(Y))\).  We obtain, for every \(\Psi\in C_{c}^{\infty}(\Omega;C_{\mathrm{per}}^{\infty}(Y)^{N})\),
\begin{equation}
\int_{\Omega}\int_{Y^{*}}
\sum_{i=1}^{N}
\bigl[q_{i}-A_{i}(y,\Psi)\bigr]
\bigl[\Xi_{i}-\Psi_{i}\bigr]\,dy\,dx
\geq0.
\label{eq:limit-Minty}
\end{equation}

By Lemma~\ref{lem:smooth-periodic-density}, every \(V\in L^{\Phi}(\Omega\times Y^{*})^{N}\) is the limit in \(L^{\Phi}(\Omega\times Y^{*})^{N}\) of functions \(\Psi_{n}\in C_{c}^{\infty}(\Omega;C_{\mathrm{per}}^{\infty}(Y)^{N})\), and \(A(y,\Psi_{n})\to A(y,V)\) in \(L^{\widetilde\Phi}(\Omega\times Y^{*})^{N}\) by Lemma~\ref{lem:nonlinear-modular-continuity}.  Since \(q\in L^{\widetilde\Phi}(\Omega\times Y^{*})^{N}\) and \(\Xi\in L^{\Phi}(\Omega\times Y^{*})^{N}\), H\"older's inequality \eqref{eq:Orlicz-Holder} extends \eqref{eq:limit-Minty} to every \(V\in L^{\Phi}(\Omega\times Y^{*})^{N}\).  For \(Z\in L^{\Phi}(\Omega\times Y^{*})^{N}\) and \(t>0\), the choice \(V=\Xi-tZ\) and division by \(t\) give
\[
\int_{\Omega}\int_{Y^{*}}\bigl[q-A(y,\Xi-tZ)\bigr]\cdot Z\,dy\,dx\geq0.
\]
Letting \(t\downarrow0\), with \(A(y,\Xi-tZ)\to A(y,\Xi)\) in \(L^{\widetilde\Phi}(\Omega\times Y^{*})^{N}\) by Lemma~\ref{lem:nonlinear-modular-continuity}, and then replacing \(Z\) by \(-Z\), we obtain \(\int_{\Omega}\int_{Y^{*}}[q-A(y,\Xi)]\cdot Z\,dy\,dx=0\) for every such \(Z\).  Consequently,
\begin{equation}
q_{i}(x,y)
=
A_{i}\bigl(y,\Xi(x,y)\bigr)
\quad\text{for almost every }(x,y)
\text{ and every }i=1,\ldots,N.
\label{eq:identified-flux}
\end{equation}
This proves the flux convergence
\eqref{eq:intro-flux-convergence}.

The cell equation identifies the corrector and completes the macroscopic limit.
Substituting \eqref{eq:identified-flux} into
\eqref{eq:pointwise-unidentified-cell} gives, for almost every
\(x\in\Omega\) and every \(\zeta\in W_{\#}^{1}L^{\Phi}(Y^{*})\),
\begin{equation}
\int_{Y^{*}}
\sum_{i=1}^{N}
A_{i}\bigl(
y,
\nabla u_{0,\alpha}(x)
+\nabla_{y}u_{1,\alpha}(x,y)
\bigr)
D_{y_{i}}\zeta(y)\,dy
=0.
\label{eq:limit-cell-equation}
\end{equation}
Uniqueness in Proposition~\ref{prop:cell-minimizer} yields
\[
u_{1,\alpha}(x,\cdot)
=
w_{\nabla u_{0,\alpha}(x)}
\quad\text{for almost every }x\in\Omega.
\]
Average \eqref{eq:identified-flux} over \(Y^{*}\), use
\eqref{eq:Fhom-derivative}, and substitute into
\eqref{eq:unidentified-macro-equation}.  We obtain, for every \(v\in W_{0}^{1}L^{\Phi}(\Omega)\),
\[
\int_{\Omega}
\sum_{i=1}^{N}
A_{\mathrm{hom},i}(\nabla u_{0,\alpha})
D_{x_{i}}v\,dx
=
\lambda_{*}\beta_{\rho}
\int_{\Omega}
\phi(\lvert u_{0,\alpha}\rvert)
\operatorname{sgn}(u_{0,\alpha})v\,dx.
\]
Thus
\[
\lambda_{*}
=
\lambda_{1,0}(\alpha)
\]
for the limit minimizer \(u_{0,\alpha}\), proving
\eqref{eq:intro-lambda-convergence} along the extracted subsequence.

The convergence of the minima was obtained without extraction.  If the nonnegative normalized limit minimizer is unique, every subsequence has the same limit \(u_{0,\alpha}\), the same corrector and the same multiplier.  The usual subsequence principle then gives convergence of the entire family.  This completes the proof.
\end{proof}

\subsection{Benchmark models}
\label{subsec:benchmarks}

Throughout this subsection we take
\[
\tau\in C_{\mathrm{per}}(Y),\qquad 0<\tau_-\le \tau(y)\le \tau_+,
\qquad F(y,\xi)=\tau(y)\Phi(|\xi|).
\]
If \(\Phi\) is a strictly convex \(N\)-function, of class \(C^{1}\) on \([0,\infty)\) and \(C^{2}\) on \((0,\infty)\), and if, for every \(t>0\),
\[
1<\underline{p}\le \frac{t\phi(t)}{\Phi(t)}\le \overline{p}<\infty,
\qquad
\frac{t\phi'(t)}{\phi(t)}\le \overline{p}-1,
\]
then \((H_1)\) holds with \(p^-=\underline{p}\) and \(p^+=\kappa^+=\overline{p}\).  Moreover \(F(y,\cdot)\) is even, strictly convex and of class \(C^{1}\), with
\(\tau_-\Phi(|\xi|)\le F(y,\xi)\le \tau_+\Phi(|\xi|)\), so that \((H_2)\) holds; the flux is
\(A(y,\xi)=\tau(y)\phi(|\xi|)\xi/|\xi|\) for \(\xi\ne0\) and \(A(y,0)=0\).  For the radial examples below it is therefore enough to compute the two logarithmic derivatives of \(\Phi\).

\subsubsection{The homogeneous \(p\)-model}

Let \(\Phi(t)=t^p/p\), \(1<p<\infty\).  Then \(\phi(t)=t^{p-1}\),
\(A(y,\xi)=\tau(y)|\xi|^{p-2}\xi\), and
\(p^-=p^+=\kappa^+=p\).  Testing the variational equation \eqref{eq:intro-weak} with
\(u_{\varepsilon,\alpha}\) gives
\(p\mu_{1,\varepsilon}(\alpha)=\lambda_{1,\varepsilon}(\alpha)\,p\alpha\), hence
\(\lambda_{1,\varepsilon}(\alpha)=\mu_{1,\varepsilon}(\alpha)/\alpha\).
Since \(\mathcal E_\varepsilon(tv)=t^p\mathcal E_\varepsilon(v)\) and
\(\mathcal M_\varepsilon(tv)=t^p\mathcal M_\varepsilon(v)\), rescaling
competitors between the levels \(1\) and \(\alpha\) gives
\[
\mu_{1,\varepsilon}(\alpha)=\alpha\mu_{1,\varepsilon}(1),
\qquad
\lambda_{1,\varepsilon}(\alpha)=\mu_{1,\varepsilon}(1).
\]
Thus the multiplier is independent of the mass \(\alpha\).

\subsubsection{A nonhomogeneous double-power model}

Let
\(\Phi(t)=t^p/p+t^r/r\) with \(1<p<r<\infty\).  Then
\(\phi(t)=t^{p-1}+t^{r-1}\) and
\(A(y,\xi)=\tau(y)\bigl(|\xi|^{p-2}+|\xi|^{r-2}\bigr)\xi\).
The quotients
\[
\frac{t\phi(t)}{\Phi(t)}
=\frac{t^{p}+t^{r}}{t^{p}/p+t^{r}/r},
\qquad
\frac{t\phi'(t)}{\phi(t)}
=\frac{(p-1)t^{p-1}+(r-1)t^{r-1}}
{t^{p-1}+t^{r-1}}
\]
are weighted averages of \(\{p,r\}\) and \(\{p-1,r-1\}\), respectively.  Hence
\((H_1)\) holds with \(p^-=p\) and \(p^+=\kappa^+=r\).
In \eqref{eq:intro-multiplier-ratio}, the numerator becomes \(\int_{\Omega_\varepsilon}\tau(x/\varepsilon)(|\nabla u_{\varepsilon,\alpha}|^p+|\nabla u_{\varepsilon,\alpha}|^r)\,dx\) and the denominator \(\varepsilon\int_{\Gamma_\varepsilon}\rho(x/\varepsilon)(|u_{\varepsilon,\alpha}|^p+|u_{\varepsilon,\alpha}|^r)\,d\sigma_x\).  In general \(\lambda_{1,\varepsilon}(\alpha)\ne\mu_{1,\varepsilon}(\alpha)/\alpha\), because amplitude scaling changes the relative weights of the two powers.

\subsubsection{A logarithmic perturbation}

Let \(p>1\) and
\(\phi(t)=t^{p-1}\log(e+t)\), \(\Phi(t)=\int_0^t\phi(s)\,ds\).  Then
\[
\frac{t\phi'(t)}{\phi(t)}
=p-1+\frac{t}{(e+t)\log(e+t)}\in[p-1,p],
\]
and integration of these differential inequalities gives
\(p\le t\phi(t)/\Phi(t)\le p+1\).  Thus \((H_1)\) holds with
\(p^-=p\) and \(p^+=\kappa^+=p+1\).

\begingroup
\hbadness=2000
\section*{Declarations}

\noindent\textbf{Data availability.} No data were generated or analysed in this study.

\smallskip
\noindent\textbf{Conflict of interest.} The authors declare that they have no conflict of interest.

\smallskip
\noindent\textbf{Funding.} The authors received no specific funding for this work.

\endgroup

\end{document}